\documentclass[11pt]{amsart}
\usepackage{newtxmath,microtype}
\usepackage[T1]{fontenc}
\usepackage{amsmath, amscd, amsthm} 
\usepackage{amsfonts} 
\usepackage{stmaryrd}
\usepackage{enumerate}
\usepackage[svgnames]{xcolor} 
\usepackage{color}
\usepackage[margin=1.5in]{geometry}
\usepackage{booktabs}
\usepackage{aliascnt}
\usepackage{hyperref}
\definecolor{dark-red}{rgb}{0.4,0.15,0.15}
\hypersetup{
	colorlinks, linkcolor=dark-red,
	citecolor=DarkBlue, urlcolor=MediumBlue
}
\usepackage{mathrsfs}
\usepackage{comment}
\usepackage[
textwidth=3cm,
textsize=small,
colorinlistoftodos]
{todonotes}

\usepackage[all]{xy}

\usepackage{tikz}
\usepackage{tikz-cd}
\usetikzlibrary{decorations.pathreplacing}

\newcommand{\Z}{\mathbb{Z}}

\newcommand{\RR}{\mathcal{R}}

\newcommand{\FF}{\mathbb{F}}

\renewcommand{\setminus}{\smallsetminus}
\newcommand{\lrangle}[1]{\langle #1 \rangle}

\DeclareMathAlphabet{\mathcal}{OMS}{cmsy}{m}{n}

\newcommand{\Sub}{\operatorname{Sub}}

\usepackage[T1]{fontenc}
\numberwithin{equation}{section} %Fiddles with numbering system of the following.

\theoremstyle{plain}

\newaliascnt{theorem}{equation}  
\newtheorem{theorem}[theorem]{Theorem}  
\aliascntresetthe{theorem}

\newaliascnt{dodeca}{equation}  
  
\aliascntresetthe{dodeca}

\theoremstyle{definition}

\newaliascnt{prop}{equation}  
\newtheorem{prop}[prop]{Proposition}
\aliascntresetthe{prop}

\newaliascnt{lemma}{equation}  
\newtheorem{lemma}[lemma]{Lemma}
\aliascntresetthe{lemma}

\newaliascnt{corollary}{equation}  
\newtheorem{corollary}[corollary]{Corollary}
\aliascntresetthe{corollary}

\newaliascnt{claim}{equation}  

\aliascntresetthe{claim}

\newaliascnt{conjecture}{equation}  
\newtheorem{conjecture}[conjecture]{Conjecture}
\aliascntresetthe{conjecture}

\newaliascnt{question}{equation}  

\aliascntresetthe{question}

\newaliascnt{defn}{equation}  
\newtheorem{defn}[defn]{Definition}
\aliascntresetthe{defn}

\newaliascnt{warn}{equation}  
\newtheorem{warn}[warn]{Warning}
\aliascntresetthe{warn}

\newaliascnt{example}{equation}  
\newtheorem{example}[example]{Example}
\aliascntresetthe{example}

\theoremstyle{remark}

\newaliascnt{remark}{equation}  
\newtheorem{remark}[remark]{Remark}
\aliascntresetthe{remark}

\newaliascnt{convention}{equation}  

\aliascntresetthe{convention}

\theoremstyle{plain}
\newtheorem*{mainthm}{Theorem}

\definecolor{gBlue}{HTML}{2b83ba}
\definecolor{gRed}{HTML}{ff5100}
\newcommand{\aref}[1]{\autoref{#1}}

\newcommand{\im}{\operatorname{im}}

\newcommand{\SL}{\mathsf{SL}}
\newcommand{\PSL}{\mathsf{PSL}}
\newcommand{\Dic}{\mathsf{Dic}}

\usepackage{stmaryrd}

\begin{document}
	%\title{$N_\infty$ operads for lossless groups}
	%\title{Lifting $N_\infty$ operads from $\Sub(G)/G$}
	\title{Lifting $N_\infty$ operads from conjugacy data}
	
	\author{Scott Balchin}
	\address{Max Planck Institute for Mathematics}
	\email{balchin@mpim-bonn.mpg.de}
	
	\author{Ethan MacBrough}
	\address{Reed College}
	\email{emacbrough@reed.edu}
	
	\author{Kyle Ormsby}
	\address{Reed College / University of Washington}
	\email{ormsbyk@reed.edu \textnormal{/} ormsbyk@uw.edu}
	
	\begin{abstract}
		We isolate a class of groups --- called \emph{lossless groups} --- for which homotopy classes of $G$-$N_\infty$ operads are in bijection with certain restricted transfer systems on the poset of conjugacy classes $\Sub(G)/G$. 
	\end{abstract}

\maketitle
	
	\tableofcontents
	
	\section{Introduction}

	The concept of an $N_\infty$ operad, as introduced by Blumberg--Hill in~\cite{BlumbergHill}, provides an equivariant analogue of $E_\infty$ operads which captures multiplicative norm maps on equivariant commutative ring spectra. Further work by various authors proved that homotopy category of $N_\infty$ operads can be identified with far simpler structures called indexing systems~\cite{BP21,GW18, Rubin2}. This was then distilled into the identification of \emph{transfer systems} which are purely combinatorial representations of $N_\infty$ operads on the subgroup lattice of the group in question~\cite{bbr,rubin}. This has led to computable approaches to understand the structures of $N_\infty$ operads for a given group.
	
	In~\cite{bbr}, the collection of $N_\infty$ operads for the cyclic groups $C_{p^n}$ were classified via the use of transfer systems on the lattice $\Sub(C_{p^n}) \cong [n]$. This concept was abstracted in~\cite{fooqw} where the notion of transfer systems was developed for an arbitrary finite poset, and shown in particular to be in bijection with weak factorization systems when the poset was moreover a complete lattice. This approach has already provided fruitful results such as classifications of model structures on total orders in~\cite{boor}.
	
	This abstraction provides a strict generalization. Indeed, in~\cite{jakovlev}, those lattices $L$ which arise as $\Sub(G)$ for an arbitrary group $G$ are classified (this was refined to the Abelian case in~\cite{contiu}). For example, the lattice in \aref{fig:notasub} does not appear as  $\Sub(G)$ for any $G$.
	\begin{figure}[h]
		\centering
		\begin{tikzpicture}
		\node[circle,draw=black, fill=black, inner sep=0pt, minimum size=5pt] (10) at (1,0) {};
		\node[circle,draw=black, fill=black, inner sep=0pt, minimum size=5pt] (01) at (0,1) {};
		\node[circle,draw=black, fill=black, inner sep=0pt, minimum size=5pt] (21) at (2,1) {};
		\node[circle,draw=black, fill=black, inner sep=0pt, minimum size=5pt] (12) at (1,2) {};
		\node[circle,draw=black, fill=black, inner sep=0pt, minimum size=5pt] (13) at (1,3) {};
		\draw[thin,black!30] (10) -- (01);
		\draw[thin,black!30] (10) -- (21);
		\draw[thin,black!30] (21) -- (12);
		\draw[thin,black!30] (01) -- (12);
		\draw[thin,black!30] (12) -- (13);
		\end{tikzpicture}
		\caption{A lattice which is not of the form $\Sub(G)$ for any $G$.}\label{fig:notasub}
	\end{figure}
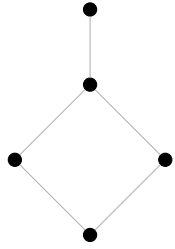
	
	In this paper, we move back to the roots of the classification of $N_\infty$ operads, with a view to improve the computational tools available. In particular, we will make the first serious foray into the realm of non-Abelian groups. In the definition of $N_\infty$ operads, conjugation must be taken into account. One may naively hope that this is superfluous, and that one can just work with transfer systems on the poset $\Sub(G)/G$, that is, the collection of subgroups up to conjugacy. This, however, does not work in general.
	
	Indeed, this is already noted in an example of Rubin~\cite{rubin}, as we now recall. If one considers the symmetric group $\mathfrak S_4$, then there are three conjugate copies of $D_4$ (the dihedral group of order $8$) living inside it. Moreover, there are three double-transpositions in $\mathfrak S_4$ which generate three conjugate copies of $C_2$. It follows that to define a transfer system for $\mathfrak S_4$ it is not enough to just declare that we have the relation $C_2 \,\RR \, D_4$, we must also keep track of which copies of $C_2$ are related to which copies of $D_4$, something that is lost when working up to conjugacy.
	
	There are cases, however, where it is possible to work up to conjugacy; again, such an example is observed by Rubin~\cite{rubin}. Moving down in the world of symmetric groups, consider $G=\mathfrak S_3$.  The subgroup lattice here takes the form displayed in \aref{fig:subg}.
	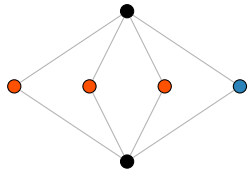
\begin{figure}[h]
		\centering
		\begin{tikzpicture}
		\node[circle,draw=black, fill=black, inner sep=0pt, minimum size=5pt] (10) at (1.5,0) {};
		\node[circle,draw=black, fill=gRed, inner sep=0pt, minimum size=5pt] (01) at (0,1) {};
		\node[circle,draw=black, fill=gRed, inner sep=0pt, minimum size=5pt] (11) at (1,1) {};
		\node[circle,draw=black, fill=gRed, inner sep=0pt, minimum size=5pt] (21) at (2,1) {};
		\node[circle,draw=black, fill=gBlue, inner sep=0pt, minimum size=5pt] (31) at (3,1) {};
		\node[circle,draw=black, fill=black, inner sep=0pt, minimum size=5pt] (02) at (1.5,2) {};
		\draw[thin,black!30] (10) -- (01);
		\draw[thin,black!30] (10) -- (11);
		\draw[thin,black!30] (10) -- (21);
		\draw[thin,black!30] (10) -- (31);
		\draw[thin,black!30] (02) -- (01);
		\draw[thin,black!30] (02) -- (11);
		\draw[thin,black!30] (02) -- (21);
		\draw[thin,black!30] (02) -- (31);
		\end{tikzpicture}
		\caption{The lattice $\Sub(G)$ for $G=\mathfrak S_3$. A red node is a copy of $C_2$, while the blue node is a copy of $C_3$.}\label{fig:subg}
	\end{figure}
	
	The key point is that we can treat all copies of $C_2$ as essentially being the same subgroup. We refer the reader to \aref{defn:gtran} for the definition of a transfer system for what follows. Suppose that we have a transfer system with the relation $1 \,\RR \, \lrangle{(12)}$. As transfer systems are required to be closed under conjugation, this implies that we necessarily have $1 \,\RR \, \lrangle{\tau}$ for every transposition $\tau$. Dually if we have $\lrangle{(12)} \,\RR \, \mathfrak S_3$ then we also have $\lrangle{\tau} \,\RR \, \mathfrak S_3$ for every transposition $\tau$. Next, we use that fact that in a transfer system if we have $H \,\RR \, K$ and $L \,\RR \, K$ then we also have $(H \cap L) \,\RR \, K$. In particular, if we  remember  that we really have three distinct copies of $C_2$, it follows that if we have $\lrangle{(12)} \,\RR \, \mathfrak S_3$, then we also have $\lrangle{(23)} \,\RR \, \mathfrak S_3$, and as such we have $1 \,\RR \, \mathfrak S_3$ by this intersection property.   
	
%	That is, if we have a transfer system with $1 \to \lrangle{(12)}$, then we necessarily have $1 \to \lrangle{\tau}$ for every transposition $\tau$, and dually if we have $\lrangle{(12)} \to \mathfrak S_3$. However, we then must remember  that we really have three distinct copies of $C_2$; it follows that if we have $\lrangle{(12)} \to \mathfrak S_3$, then we also have $1 \to \mathfrak S_3$ as by conjugation it contains $\lrangle{(23)} \to \mathfrak S_3$, and then we close under restriction (i.e., $\lrangle{(12)} \cap \lrangle{(23)} = 1$). 
%	
	It turns out that this is the only condition that one needs to impose in this case. So we may do exactly as we want, study transfer systems on $\Sub(G)/G$, which is of the form \aref{fig:subgmodg}.
	\begin{figure}[h]
		\centering
		\begin{tikzpicture}
		\node[circle,draw=black, fill=black, inner sep=0pt, minimum size=5pt] (10) at (1.5,0) {};
		\node[circle,draw=black, fill=gRed, inner sep=0pt, minimum size=5pt] (01) at (1,1) {};
		\node[circle,draw=black, fill=gBlue, inner sep=0pt, minimum size=5pt] (31) at (3,1) {};
		\node[circle,draw=black, fill=black, inner sep=0pt, minimum size=5pt] (02) at (1.5,2) {};
		\draw[thin,black!30] (10) -- (01);
		\draw[thin,black!30] (10) -- (31);
		\draw[thin,black!30] (02) -- (01);
		\draw[thin,black!30] (02) -- (31);
		\end{tikzpicture}
		\caption{The lattice $\Sub(G)/G$ for $G=\mathfrak S_3$. The red node is a copy of $C_2$ up to conjugacy, while a blue node is the copy of $C_3$.}\label{fig:subgmodg}
	\end{figure}
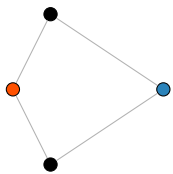
	
	Any transfer system on this lattice (of which there are 10) lifts to a transfer system for $\Sub(G)$ when we additionally satisfy that whenever we have $\begin{tikzpicture}	\node[circle,draw=black, fill=gRed, inner sep=0pt, minimum size=5pt] (01) at (0,0) {};\end{tikzpicture} \,\RR \, \mathfrak S_3$ we also have $1 \,\RR \, \mathfrak S_3$. It turns out that there are 9 such transfer systems.
	
	The goal of this paper is to isolate a class of groups for which $G$-transfer systems can be explicitly characterized as certain restricted transfer systems on the poset $\Sub(G)/G$ in this fashion. The class of groups that we will isolate here are the \emph{lossless groups} (see \aref{defn:lossless}). Not only does this provide a non-trivial structural result, it also equips us with powerful computational tools to classify $N_\infty$ operads for a wide range of non-Abelian groups, and provides a better conceptual understanding of the structures involved.

	For example, let us assume that we wish to study the collection of $G$-$N_\infty$ operads for $G=D_9$. The lattice that we need to consider, $\Sub(D_9)$, is displayed in \aref{fig:d9}.
	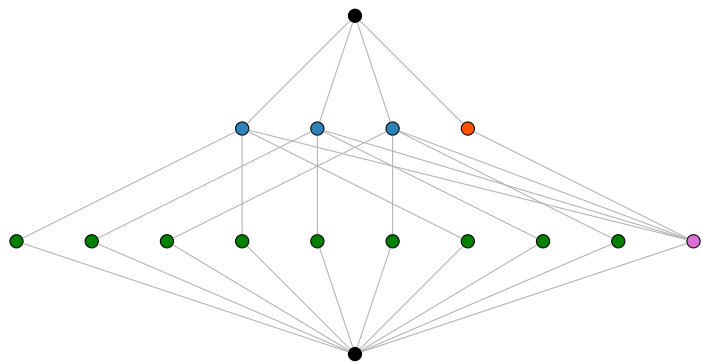
\begin{figure}[h]	
		\centering
		\begin{tikzpicture}[yscale=1.5,xscale=-1]
		\node[circle,draw=black, fill=black, inner sep=0pt, minimum size=5pt] (00) at (4.5,0) {};
		\node[circle,draw=black, fill=Orchid, inner sep=0pt, minimum size=5pt] (01) at (0,1) {};
		\node[circle,draw=black, fill=Green, inner sep=0pt, minimum size=5pt] (11) at (1,1) {};
		\node[circle,draw=black, fill=Green, inner sep=0pt, minimum size=5pt] (21) at (2,1) {};
		\node[circle,draw=black, fill=Green, inner sep=0pt, minimum size=5pt] (31) at (3,1) {};
		\node[circle,draw=black, fill=Green, inner sep=0pt, minimum size=5pt] (41) at (4,1) {};
		\node[circle,draw=black, fill=Green, inner sep=0pt, minimum size=5pt] (51) at (5,1) {};
		\node[circle,draw=black, fill=Green, inner sep=0pt, minimum size=5pt] (61) at (6,1) {};
		\node[circle,draw=black, fill=Green, inner sep=0pt, minimum size=5pt] (71) at (7,1) {};
		\node[circle,draw=black, fill=Green, inner sep=0pt, minimum size=5pt] (81) at (8,1) {};
		\node[circle,draw=black, fill=Green, inner sep=0pt, minimum size=5pt] (91) at (9,1) {};
		\node[circle,draw=black, fill=gRed, inner sep=0pt, minimum size=5pt] (32) at (3,2) {};
		\node[circle,draw=black, fill=gBlue, inner sep=0pt, minimum size=5pt] (42) at (4,2) {};
		\node[circle,draw=black, fill=gBlue, inner sep=0pt, minimum size=5pt] (52) at (5,2) {};
		\node[circle,draw=black, fill=gBlue, inner sep=0pt, minimum size=5pt] (62) at (6,2) {};
		\node[circle,draw=black, fill=black, inner sep=0pt, minimum size=5pt] (03) at (4.5,3) {};
		\foreach \xtick in {0,...,9} \draw[thin,black!30] (00) -- (\xtick1);
		\draw[thin,black!30] (91) -- (62);
		\draw[thin,black!30] (81) -- (52);
		\draw[thin,black!30] (71) -- (42);
		\draw[thin,black!30] (61) -- (62);
		\draw[thin,black!30] (51) -- (52);
		\draw[thin,black!30] (41) -- (42);
		\draw[thin,black!30] (31) -- (62);
		\draw[thin,black!30] (21) -- (52);
		\draw[thin,black!30] (11) -- (42);
		\draw[thin,black!30] (01) -- (62);
		\draw[thin,black!30] (01) -- (52);
		\draw[thin,black!30] (01) -- (42);
		\draw[thin,black!30] (01) -- (32);
		\draw[thin,black!30] (62) -- (03);
		\draw[thin,black!30] (52) -- (03);
		\draw[thin,black!30] (42) -- (03);
		\draw[thin,black!30] (32) -- (03);
		\end{tikzpicture}
		\caption{The lattice $\Sub(G)$ for $G=D_{9}$. The colored nodes in the rows indicate different conjugacy classes of subgroups.}\label{fig:d9}
	\end{figure}

Hoping to find patterns or structure on this lattice is a daunting task. However, the results presented here will allow us to instead explore structure in the much more manageable --- and human friendly --- $\Sub(G)/G$ as displayed in \aref{fig:d9mod}.  In~\cite{bmodihedral} we shall undertake this exploration, and provide a recursive algorithm for computing $N_\infty$ operads for the dihedral groups $D_{p^n}$ for all $n \geqslant 0$ ($p \neq 2$).

	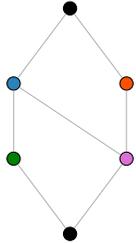
\begin{figure}[h]
		\centering
		\begin{tikzpicture}[xscale=1.5]
		\node[circle,draw=black, fill=black, inner sep=0pt, minimum size=5pt] (00) at (0.5,0) {};
		\node[circle,draw=black, fill=Green, inner sep=0pt, minimum size=5pt] (01) at (0,1) {};
		\node[circle,draw=black, fill=Orchid, inner sep=0pt, minimum size=5pt] (11) at (1,1) {};
		\node[circle,draw=black, fill=gBlue, inner sep=0pt, minimum size=5pt] (02) at (0,2) {};
		\node[circle,draw=black, fill=gRed, inner sep=0pt, minimum size=5pt] (12) at (1,2) {};
		\node[circle,draw=black, fill=black, inner sep=0pt, minimum size=5pt] (03) at (0.5,3) {};
		\draw[thin,black!30] (00) -- (01);
		\draw[thin,black!30] (00) -- (11);
		\draw[thin,black!30] (01) -- (02);
		\draw[thin,black!30] (11) -- (02);
		\draw[thin,black!30] (11) -- (12);
		\draw[thin,black!30] (02) -- (03);
		\draw[thin,black!30] (12) -- (03);
		\end{tikzpicture}
		\caption{The lattice $\Sub(G)/G$ for $G=D_9$. The coloring of the nodes corresponds to the coloring in \autoref{fig:d9}.}\label{fig:d9mod}
	\end{figure}

	In \autoref{sec:transfergroups} we will introduce the main object of study of this paper, the \emph{lossless groups}, and prove that there is a bijection between $G$-transfer systems for a lossless group $G$ and \emph{liftable transfer systems} on $\Sub(G)/G$. We record the main theorem here.
	
	\begin{mainthm}[\aref{cor:losslesssummary}]
		Let $G$ be a lossless group. Then there is a bijection between homotopy classes of $G$-$N_\infty$ operads and liftable transfer systems on $\Sub(G)/G$.
	\end{mainthm}

	In fact, in a sense made precise in \autoref{sec:transfergroups}, the above theorem characterizes lossless groups.
	
	We then continue in \autoref{sec:families} and prove that many groups of interest are in fact lossless. In \autoref{sec:metacyclic} we direct our attention to a particularly nice class of lossless groups, namely the \emph{metacyclic Frobenius groups}. For these groups we have a detailed understanding of both the form of $\Sub(G)/G$ and the lifting conditions required to determine a $G$-transfer system.  We apply this theory in \autoref{sec:examples} to demonstrate how the theory aids computations. Finally, in \autoref{sec:lossy} we outline a potential strategy for dealing with lossy groups. In particular we shall focus on the case of $G = SL_2(\mathbb{F}_p)$ which is an important family of groups in the study of topological modular forms with level structures~\cite{tmf}.
	
	\subsection*{Conventions}
	
	Throughout, we shall use the following conventions for group theory:
	
	\begin{itemize}
		\item $H\leqslant G$ designates $H$ as a subgroup of $G$.
		\item $N\trianglelefteqslant G$ designates $N$ as a normal subgroup.
		\item If $g\in G$ then ${}^gH := gHg^{-1}$ (thus ${}^g({}^hH) = {}^{(gh)}H$ for all $g,h\in G$).
		\item For $G$ a group we write $\Sub(G)/G$ for the poset of conjugacy classes of subgroups of $G$.
		\item For $n > 2$, we write $D_n$ for the dihedral group of order $2n$.
		%\item For $n > 1$, we write $\Dic_n $ for the dicyclic group of order $4n$.
	\end{itemize}
	
	\subsection*{Acknowledgements}

	The first author would like to thank the Max Planck Institute for Mathematics for its hospitality, and was partially supported by the European Research Council (ERC) under Horizon Europe (grant No.~101042990). The second author thanks Coil Technologies for their generous donation to fund his tuition, which enabled him to conduct this research. The third author's work was supported by the National Science Foundation under Grant No.~DMS-2204365. The authors thank the anonymous referee for helpful comments and suggestions.

	%\part{Lossless groups}
	
	\section{Transfer systems on lossless groups}\label{sec:transfergroups}
	
	In this section we will introduce the class of \emph{lossless groups}, which allow us to study $G$-transfer systems using only categorical transfer systems on $\Sub(G)/G$. After proving the basic structural results about these groups, we identify several simple criteria for when a group is lossless, and also provide several examples of how groups can fail to be lossless.
	
	\subsection{General results}
	
	We first recall the definitions of $G$-transfer systems and categorical transfer systems. We refer the reader to~\cite{bbr, fooqw} for further details.
	
	\begin{defn}\label{defn:gtran}
		Let $G$ be a finite group. A ($G$-)\emph{transfer system} is a relation $\RR$ on $\Sub(G)$ refining inclusion satisfying the following:
		\begin{itemize}
			\item (reflexivity) $H \,\RR\, H$ for all $H \leqslant G$;
			\item (transitivity) $K \, \RR \, H$ and $L\,  \RR \, K$ implies $L\, \RR \,H$;
			\item (closed under conjugation) $K\, \RR \,H$ implies that ${}^gK\, \RR \,{}^gH$ for all $g \in G$;
			\item (closed under restriction) $K\, \RR \,H$ and $L \leqslant H$ implies $(K {\cap} L) \, \RR \, L$.
		\end{itemize}
	\end{defn}
	
	\begin{prop}[\cite{bbr}]
	Let $G$ be a finite group. Then there is a bijection between the set of transfer systems on $\Sub(G)$ and the set of $G$-$N_\infty$ operads.
	\end{prop}
	
	In \cite{fooqw}, a notion of an abstract categorical transfer system was introduced for an arbitrary poset, but with a particular focus on when the poset in question is a lattice. In fact, despite presenting the definition for general posets, the authors only make serious use of their definition in the setting of lattices. Since $\Sub(G)/G$ may not be a lattice when $G$ is non-commutative, we are forced to think seriously about more general posets. The definition we present below is \emph{not} equivalent to the definition given in \cite{fooqw}, since we require restriction closure for arbitrary maximal lower bounds, rather than just when a unique meet exists. However, it is straightforward to verify the two definitions coincide when the poset is a lattice. For an element $x$ in some poset $\mathcal{P}$, let $x^{\downarrow}$ denote the down-set of $x$ in $\mathcal{P}$, i.e. the set of all $y\leqslant x$.

\begin{defn}\label{defn:cattran}
Let $\mathcal{P} = (\mathcal{P}, \leqslant)$ be a poset. A (categorical) \emph{transfer system} on $\mathcal{P}$ consists of a partial order $\RR$ on $\mathcal{P}$ that refines $\leqslant$ and such that whenever $x \, \RR \, y$ and $z\leqslant y$, then for all maximal $w\in x^{\downarrow}\cap z^{\downarrow}$ we have $w \, \RR \, z$.
\end{defn}

\begin{remark}
Using this definition, Theorem 4.13 in \cite{fooqw} (categorical transfer systems on a lattice are in natural bijection with weak factorization systems) can be generalized to arbitrary posets. This gives some evidence that our definition is the ``morally correct'' one for non-lattice posets. More pragmatically, the stronger definition is necessary to make \autoref{lem:lifting} work.
\end{remark}
	
	\begin{warn}\label{warn:achtung}
		Until this point in the literature, only transfer systems on Abelian groups have been seriously considered. In this case, $G$-transfer systems are in bijection with categorical transfer systems on $\Sub(G) \cong \Sub(G)/G$, and as such there is no distinction to be made. We are primarily concerned with non-Abelian groups in this paper, and as such, one needs to be careful what they mean.
	\end{warn}
	
	We now introduce the notion of a lossless group. This definition was isolated to capture exactly the groups needed for our applications, and we have been unable to find this class of groups studied previously in the literature.
	
	\begin{defn}\label{defn:lossless}
		A \emph{lossless group} is a group $G$ such that for all pairs of subgroups $K\leqslant H$ such that ${}^gK\leqslant H$ for some $g\in G$, there exists some $h\in N_G(H)$ such that ${}^hK = {}^gK$. A group which is not lossless will be called \emph{lossy}.
	\end{defn}
	
	\begin{remark}
		An equivalent succinct way of phrasing this definition is that for all $H\leqslant G$, the fusion of subgroups of $H$ is controlled by $N_G(H)$. Note that we do not require fusion to be controlled on elements as is common in group theoretic literature. This corresponds in the previous definition to the fact that we only require $h^{-1}g\in N_G(K)$, rather than requiring $h^{-1}g\in C_G(K)$.
	\end{remark}

	In the following definition, $[H]$ denotes the conjugacy class of a subgroup $H\leqslant G$. We also adopt the notational convention of writing $K\to H\in \RR$ as shorthand for $K\,\RR\,H$; we find this notational flexibility useful, especially as the notation for a given transfer system becomes more complex.
	
	\begin{defn}
		Let $G$ be a group and $\pi\colon\Sub(G)\to\Sub(G)/G$ be the quotient map of posets.
		\begin{itemize}
			\item For a categorical transfer system $ \RR $ on $\Sub(G)/G$, we define $\pi^{-1}( \RR )$ to be the relation on $\Sub(G)$ such that $K\to H\in\pi^{-1}( \RR )$ if and only if $[K]\to[H]\in \RR $. We then define $\pi^*(\RR)$ to be the $G$-transfer system generated by $\pi^{-1}(\RR)$.
			\item For a $G$-transfer system $ \RR $, we define $\pi_*( \RR )$ to be the relation on $\Sub(G)/G$ where $[K]\to[H]\in\pi_*(\RR)$ if and only if there exist some pair of subgroups $K'\leqslant H'$ with $[K] = [K']$ and $[H] = [H']$, such that $K'\to H'\in\RR$.
		\end{itemize}
	\end{defn}

	We start with an observation that we can push any $G$-transfer system to a categorical transfer system on $\Sub(G)/G$ provided that $G$ is lossless.
	
	\begin{lemma}\label{lem:pushforward}
		Let $G$ be a lossless group. Then for all $G$-transfer systems $\RR$, $\pi_*( \RR )$ is a categorical transfer system on $\Sub(G)/G$.
	\end{lemma}
	\begin{proof}
		Let $\RR' = \pi_*(\RR)$. Suppose $[K]\, \RR ' \, [H]\, \RR' \,[L]$ and let there be lifts (i.e., representatives in the conjugacy class) $K'\, \RR \, H'$ and $H''\, \RR \, L'$. Then $[H'] = [H] = [H'']$ implies we can find $g\in G$ for which ${}^gH' = H''$. Then ${}^gK'\, \RR \, L'$ is a lift of $[K]\to[L]$. Thus $ \RR '$ is transitive. Note that this part does not use the assumption that $G$ is lossless.
		
		Now suppose $[K]\, \RR' \,[H]$ and $[L]\leqslant[H]$, and suppose $[M]$ is maximal among $[M]\leqslant[K]$ and $[M]\leqslant[L]$. We can assume without loss of generality that $K\, \RR \, H$ is a lift and $L\leqslant H$ and $M\leqslant K$. Let $g\in G$ such that ${}^gM\leqslant L$. Since $G$ is lossless we can assume $g\in N_G(H)$. Thus ${}^gK\, \RR \, H$ and ${}^gM\leqslant {}^gK\cap L$. But $[{}^gK\cap L]\leqslant [K],[L]$, so by maximality ${}^gM$ must be conjugate to ${}^gK\cap L$, and hence for order reasons we have ${}^gM = {}^gK\cap L$. But since $\, \RR \,$ is a transfer system we have ${}^gM \, \RR \, L$ as ${}^gM = ({}^gK\cap L)$, and hence $[M]\, \RR '\,[L]$.
	\end{proof}
	
	\begin{prop}\label{prop:lossless}
		If $G$ is lossless then $\pi^*\dashv\pi_*$ is a Galois connection between $G$-transfer systems and categorical transfer systems on $\Sub(G)/G$. Furthermore, the unit of this adjunction is the identity, i.e., for all $G$-transfer systems $\RR$ we have $\RR = \pi^{-1}(\pi_*(\RR)) = \pi^*(\pi_*(\RR))$.
	\end{prop}
	\begin{proof}
		Let $\RR$ be a $G$-transfer system and $\RR'$ a categorical transfer system on $\Sub(G)/G$. By definition, $\RR\geqslant \pi^*(\RR')$ if and only if $\RR\geqslant \pi^{-1}(\RR')$ if and only if $[K]\,\RR'\,[H]$ implies $K\,\RR\,H$. On the other hand, we have $\pi_*(\RR)\geqslant \RR'$ if and only if $[K]\,\RR'\,[H]$ implies there exists some $K'\leqslant H'$ with $[K'] = [K]$ and $[H'] = [H]$ such that $K'\,\RR\,H'$.
		
		Thus we need to show $K'\,\RR\,H'$ implies $K\,\RR\,H$. Or in other words, we need to show that if $G$ is lossless, then for any $G$-transfer system $\RR$ and any two pairs $K\leqslant H$, $K'\leqslant H'$ such that $[K] = [K']$ and $[H] = [H']$, we have $K\,\RR\,H$ if and only if $K'\,\RR\,H'$. By symmetry we can suppose $K\,\RR\,H$ and we want to show this implies $K'\,\RR\,H'$. But $[H] = [H']$ implies by definition we can find some $g\in G$ such that ${}^gH = H'$, and by conjugation closure we have ${}^gK\,\RR\,{}^gH$, so we might as well assume $H = H'$. Let $h\in G$ such that ${}^hK = K'$. Then we have $K,{}^hK\leqslant H$, so since $G$ is lossless we can assume $h\in N_G(H)$. But then again using conjugation closure we have $(K' = {}^hK) \,\RR\, ({}^hH = H = H')$ as claimed.
		
		From the definition of $\pi^{-1}(-)$, this also shows $\RR = \pi^{-1}(\pi_*(\RR))$ for any $G$-transfer system $\RR$. Since $\pi^*(\RR')$ is defined to be the smallest $G$-transfer system containing $\pi^{-1}(\RR')$ and $\RR$ is a $G$-transfer system by definition, this also implies $\RR = \pi^*(\pi_*(\RR))$ and hence the unit is the identity.
	\end{proof}
	
	\begin{corollary}
		If $G$ is lossless then every $G$-transfer system can be lifted from a categorical transfer system on $\Sub(G)/G$.
	\end{corollary}
	
	We note that a strong converse to this corollary also holds:
	
	\begin{prop}
		Let $G$ be an arbitrary (finite) group. If every $G$-transfer system can be written as $\pi^*(\RR)$ for some arbitrary \emph{relation} $\RR$ on $\Sub(G)/G$, then $G$ is lossless.
	\end{prop}
	\begin{proof}
		Let $S$ be some set of arrows in $\Sub(G)$, and let $\RR$ be the $G$-transfer system generated by $S$. By the explicit construction of the transfer system generated by a set of arrows given in~\cite[Appendix B]{rubin}, one can show that $K\to H\in\RR$ necessarily implies that there must exist some $K'\to H'\in S$ and some $g\in G$ such that $K\leqslant {}^gK'$ and $H\leqslant {}^gH'$.
		
		Now suppose $K,{}^gK\leqslant H$, and let $\RR$ be the $G$-transfer system generated by $K\to H$. Suppose $\RR = \pi^*(\RR')$. Since $\pi^*(\RR')$ is generated by $\pi^{-1}(\RR')$, there must be some $K'\to H'\in\pi^{-1}(\RR')$ and some $h\in G$ such that $K\leqslant {}^hK'$ and $H\leqslant {}^hH'$. But also $\RR$ is generated by $S = \{K\to H\}$, so there must exist some $k\in G$ such that $K'\leqslant {}^kK$ and $H'\leqslant {}^kH'$. For order reasons this forces $K = {}^hK'$ and $H = {}^hH'$, so $[K]\to [H] = [K']\to[H']\in\RR'$.
		
		But we also have $[{}^gK] \to [H]$ and  $[K]\to [H]\in\RR'$ and $[H] = [K]$, so by definition of $\pi^{-1}(\RR')$ we must have ${}^gK\to H\in \RR$. Thus we can find $h\in G$ such that ${}^gK\leqslant {}^hK$ and $H\leqslant {}^hH$, and again for order reasons this forces ${}^gK = {}^hK$ and $H = {}^hH$, i.e., $K$ and ${}^gK$ are conjugate by way of $h\in N_G(H)$. Since $K,{}^gK\leqslant H$ was arbitrary, this shows $G$ is lossless.
	\end{proof}
	
	\begin{remark}\label{rem:whatarelifts}
		By \autoref{prop:lossless} and the general theory of Galois connections, when $G$ is lossless we have a bijection between $G$-transfer systems and categorical transfer systems $\RR$ on $\Sub(G)/G$ such that $\RR = \pi_*(\pi^*(\RR))$. We call such transfer systems \emph{liftable}. For example, considering $G=\mathfrak S_3$, as mentioned in the introduction a categorical transfer system is liftable if and only if whenever we have $\tau \to \mathfrak S_3$ for any transposition $\tau$ we also have $1 \to \mathfrak S_3$ (c.f., \autoref{fig:subgmodg} and the surrounding discussion).
		
		In general the condition $\RR = \pi_*(\pi^*(\RR))$ is not very evocative. In \autoref{lem:lifting} we give a more concrete set of conditions for a categorical transfer system to be liftable, but for general (lossless) $G$ these conditions can still be rather opaque and difficult to verify. In \autoref{sec:metacyclic} we will consider some special cases where the lifting conditions are simple enough to be visually intuitive, like in the case of $G = \mathfrak S_3$.
	\end{remark}
	
	The following corollary summarizes the results of this section.
	
	\begin{corollary}\label{cor:losslesssummary}
    If $G$ is a lossless group, then the function $\pi^*$ from liftable transfer systems on $\Sub(G)/G$ to transfer systems for $G$ (i.e., homotopy classes of $G$-$N_\infty$ operads) is a bijection; furthermore, for a general finite group $G$, bijectivity of $\pi^*$ (when restricted to liftable transfer systems) implies that $G$ is lossless.
	\end{corollary}
	
	\subsection{Examples and counterexamples of lossless groups}\label{sec:families}
	
	We continue with some observations regarding lossless groups, and provide some non-trivial examples of interesting families of lossless groups. Note that clearly any Abelian group is lossless.
	
	In general it appears that the class of lossless groups is rather poorly behaved under group-theoretic operations. However, we can prove that lossless groups at least play nicely with quotients.
	
	\begin{lemma}\label{quot}
		Any quotient of a lossless group is lossless.
	\end{lemma}
	\begin{proof}
		This follows directly from the standard equivariant poset isomorphism between $\Sub(G/N)$ and the interval $[N,G]\subseteq \Sub(G)^N$ for all $N\trianglelefteqslant G$.
	\end{proof}

	\begin{remark}
	The product of lossless groups can be lossy. Indeed, consider the group $G = C_2 \times A_4$, for which \aref{fig:c2a4} displays $\Sub(G)/G$.

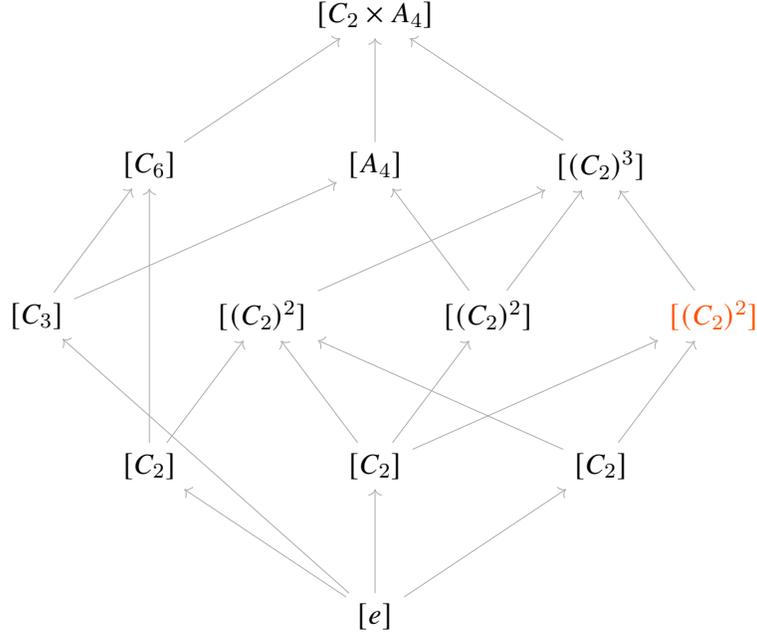
\begin{figure}[h]
\centering
\begin{tikzpicture}[xscale=1.5]
	\node[draw=none] (1) at (0, 0) {$[e]$};
	\node [draw=none] (3) at (0,2) {$[C_2]$};
	\node [draw=none] (2) at (-2,2) {$[C_2]$};
	\node [draw=none] (4) at (2,2) {$[C_2]$};
	\node [draw=none] (5) at (-3,4) {$[C_3]$};
	\node [draw=none] (6) at (1,4) {$[(C_2)^2]$};
	\node [draw=none] (7) at (-1,4) {$[(C_2)^2]$};
	\node [draw=none, color=gRed] (8) at (3,4) {$[(C_2)^2]$};
	\node [draw=none] (9) at (-2,6) {$[C_6]$};
	\node [draw=none] (11) at (0,6) {$[A_4]$};
	\node [draw=none] (10) at (2,6) {$[(C_2)^3]$};
	\node[draw=none] (12) at (0, 8) {$[C_2 \times A_4]$};
	\draw[thin,black!30,->]  (1) to (2);
	\draw[thin,black!30,->]  (1) to (3);
	\draw[thin,black!30,->]  (1) to (4);
	\draw[thin,black!30,->]  (1) to (5);
	\draw[thin,black!30,->]  (3) to (6);
	\draw[thin,black!30,->]  (2) to (7);
	\draw[thin,black!30,->]  (3) to (7);
	\draw[thin,black!30,->]  (4) to (7);
	\draw[thin,black!30,->]  (3) to (8);
	\draw[thin,black!30,->]  (4) to (8);
	\draw[thin,black!30,->]  (2) to (9);
	\draw[thin,black!30,->]  (5) to (9);
	\draw[thin,black!30,->]  (6) to (10);
	\draw[thin,black!30,->]  (7) to (10);
	\draw[thin,black!30,->]  (8) to (10);
	\draw[thin,black!30,->]  (5) to (11);
	\draw[thin,black!30,->]  (6) to (11);
	\draw[thin,black!30,->]  (9) to (12);
	\draw[thin,black!30,->]  (10) to (12);
	\draw[thin,black!30,->]  (11) to (12);
\end{tikzpicture}
\caption{The structure of $\Sub(G)/G$ for $G = C_2 \times A_4$.}\label{fig:c2a4}
\end{figure}	
	
	In $G$ we have a conjugacy class $[(C_2)^2]$ which contains only two copies of $[C_2]$ displayed in red in \aref{fig:c2a4}. Since each $(C_2)^2$ contains three copies of $C_2$, this implies the existence of some $H = (C_2)^2$ which contains two conjugate copies of $C_2$, $K$ and $^gK$. However, the normalizer of $(C_2)^2$ is $(C_2)^3$, which is Abelian and hence $K$, $^gK$ cannot be conjugate in $(C_2)^3 = N_G(H)$. That is, $C_2 \times A_4$ fails to be lossless even though $C_2$ and $A_4$ are lossless.

	Heuristically it seems very likely that being lossless does not imply that all (even normal) subgroups are lossless. Indeed, if $K,{}^gK\leqslant H\leqslant L\leqslant G$ are such that $K,{}^gK$ are not conjugate in $N_L(H)$ (so in particular $L$ is definitely not lossless), then it may still be the case that $K,{}^gK$ are conjugate in $N_G(H)$ (so $G$ might be lossless). However the authors have been unable to find an explicit example of a lossless group with a lossy subgroup.
	\end{remark}
	
	A  substantial class of examples of lossless groups comes from the following observation. Recall that a \emph{T-group} is a group in which every subnormal subgroup is normal, i.e., $K\trianglelefteqslant H$ and $H\trianglelefteqslant G$ implies $K\trianglelefteqslant G$~\cite[\S 13.4]{robinson}. A subgroup $K\leqslant G$ is \emph{pronormal} if for all $g\in G$, $K$ and ${}^gK$ are conjugate in $\lrangle{K,{}^gK}$~\cite[\S I.6]{DH92}. We note that if $K,{}^gK\leqslant H$ then $\lrangle{K,{}^gK}\leqslant H\leqslant N_G(H)$, so if every subgroup of $G$ is pronormal then $G$ is lossless. 
	
	Before we continue, let us recall the theory of Hall subgroups, and Hall's theorem, which can be seen as a generalization of Sylow's theorem in the solvable case.
	
	\begin{defn}
	A \emph{Hall subgroup} of a group $G$ is a subgroup whose order is coprime to its index.  If $\pi$ is a set of primes, then a \emph{Hall $\pi$-subgroup} is a subgroup whose order is a product of primes in $\pi$.
	\end{defn}
	
	\begin{theorem}[Hall's Theorem~\cite{hall}]
		Let $G$ be a finite solvable group and $\pi$ any set of primes. Then $G$ has a Hall $\pi$-subgroup, and any two such Hall $\pi$-subgroups are conjugate. Moreover, any subgroup whose order is a product of primes in $\pi$ is contained in some Hall $\pi$-subgroup.
	\end{theorem}
	
	\begin{prop}\label{prop:tgroup}
		Any (finite) solvable T-group is lossless.
	\end{prop}
	\begin{proof}
		A theorem of Peng tells us that a group $G$ is a solvable T-group if and only if every $p$-subgroup of $G$ is pronormal, and in this case $G$ is in fact supersolvable~\cite{peng}. We claim that in this case all subgroups of $G$ are pronormal, and hence $G$ is lossless. This result seems to be well-known among group theorists (see, e.g.,~\cite{ferrara,giovanni,kurdachenko}), but the authors of the present paper were unable to find a proof in the literature so we include one here.
		
		Suppose that $G$ is a (super)solvable T-group, and let $K\leqslant G$. We will prove by induction on the number of prime divisors of $|K|$ that $K$ is pronormal. By Peng's theorem this holds in the case where $|K|$ only has a single prime divisor. Since $G$ is supersolvable, $K\leqslant G$ is also supersolvable, so $K$ has a Sylow $p$-subgroup $P\leqslant K$ and a normal Hall $p'$-subgroup $S\trianglelefteqslant K$ for some $p$. In particular, $P\leqslant N_G(S)$. By induction we can assume $S$ and $P$ are both pronormal in $G$. By~\cite[1.8]{rose}, this implies $K = SP$ is pronormal in $G$.
	\end{proof}
	
	Although the solvable T-groups form a fairly large class of groups, this class notably excludes most interesting $p$-groups. Indeed, since every subgroup of a nilpotent group is subnormal, a $p$-group is a T-group if and only if every subgroup is normal, so the only non-Abelian $p$-groups obtained this way are groups of the form $Q_8\times (\Z/2)^n$, where $Q_8$ is the ordinary quaternion group~\cite{dedekind}. In light of this, the remainder of this section will largely focus on $p$-groups and determining conditions under which a $p$-group is lossless.
	
	The following result is trivial, but will show that \autoref{p3-example} provides a minimal example of lossy $p$-groups.
	
	\begin{prop}\label{p3}
		Let $G$ be any $p$-group of order at most $p^3$. Then $G$ is lossless.
	\end{prop}
	\begin{proof}
		Suppose $K,{}^gK\leqslant H$. We want to show $K,{}^gK$ are conjugate in $N_G(H)$. We can assume $K<H$ since otherwise $K = {}^gK$ and there's nothing to show. Since $G$ is nilpotent we have $H<N_G(H)$, and we're already done if $N_G(H) = G$, so we must have $[G:H]\geqslant p^2$. But if $|G|\leqslant p^3$ then this forces $K = \{e\} = {}^gK$ so there's nothing to show.
	\end{proof}
	
	\begin{example}\label{p3-example}
		Let $p\neq 2$ and let $N = (\Z/p)^3$. Let a generator of $T = C_p$ act on $N$ by the matrix
		\[A := \begin{bmatrix}
		1       & 1 & 1\\
		0       & 1 & 1\\
		0       & 0 & 1\\
		\end{bmatrix}.\]
		Let $G = N\rtimes T$. Note that $|G| = p^4$. Let $K = \lrangle{(0,0,1)}\subseteq N$ and $L = \lrangle{(0,0,1),(1,1,1)}\subseteq N$. Let $g\in T$ be a generator. Then $K,{}^gK\leqslant L$, but $N_G(L) = N$ is Abelian and hence $K\neq{}^gK$ cannot be conjugate in $N_G(L)$. Thus $G$ is not lossless.
	\end{example}
	
	\begin{remark}
		The reason we needed to assume $p\neq 2$ in \autoref{p3-example} is because if $p=2$ then the matrix $A$ has order $4$ instead of $2$. By an exhaustive search one can show that every group of order $2^4$ is lossless.
	\end{remark}
	
	We recall for the proof of \aref{metap} that a subgroup $H$ of $G$ is said to be \emph{characteristic} if every automorphism of $G$ fixes $H$, that is, $\phi(H)=H$ for every automorphism $\phi$ of $G$.
	
	\begin{prop}\label{metap}
		If $G$ has a cyclic normal subgroup of prime index, i.e., if $G$ is an extension
		\[1\to C_n\to G\to C_p\to 1\]
		for some prime $p$, then $G$ is lossless.
	\end{prop}
	\begin{proof}
		Let $K,{}^gK\leqslant H\leqslant G$, and let $N = C_n\trianglelefteqslant G$. Since every subgroup of $C_n$ is characteristic, all subgroups of $N$ are normal in $G$. Thus if $K\leqslant H\cap N$ then $K$ is normal so $K = {}^gK$ and there's nothing to show. Otherwise the quotient $H\to H/{H\cap N}\cong C_p$ is non-trivial restricted to $K\leqslant H$, and hence $H = K(H\cap N)$. Since $H\cap N\leqslant N$ is normal in $G$ this implies
		\[{}^gH = {}^g(K(H\cap N)) = {}^gK{}^g(H\cap N) = {}^gK(H\cap N)\leqslant H\]
		and hence $g\in N_G(H)$.
	\end{proof}
	
	We will now wish to discuss some families of groups which are amenable to the above result. We will define some of the groups in question as they may not be standard knowledge. From their description via generators and relations it is clear that they all have cyclic normal subgroups of order 2.
	\begin{defn}\label{defn:groups}\leavevmode
	\begin{itemize}
		\item The \emph{dicyclic group} of order $4n$, denoted $\mathsf{Dic}_n$ is defined via generators and relations as
		\[
			\mathsf{Dic}_n := \langle r,s \mid r^{2n} = s^4 = 1, srs = r^{2n-1}  \rangle.
		\]
		\item The \emph{semidihedral group} of order $2^n$, denoted $\mathsf{SD}_n$ is defined via generators and relations as
		\[
			\mathsf{SD}_n := \langle r,s \mid r^{2^{n-1}} =s^2 = 1 , srs = r^{2^{n-2}-1} \rangle.
		\]
		\item The \emph{modular maximal-cyclic group} of order $2^n$, denoted $\mathsf{MM}_n$ is defined via generators and relations as
		\[
			\mathsf{MM}_n := \langle r,s \mid r^{2^{n-1}} =s^2 = 1 , srs = r^{2^{n-2}+1} \rangle.
		\]
	\end{itemize}
	\end{defn}
	%We refer the reader to~\cite{gorenstein} for relevant definitions.
	
	\begin{corollary}\label{2}
		Any dihedral group, dicyclic (e.g., generalized quaternion) group, semidihedral group, or modular maximal-cyclic group is lossless.
	\end{corollary}
	
	\begin{corollary}
		If $q$ is a prime power with $q\equiv 3\mod 4$, then the Sylow $2$-subgroup of $\PSL_3(\FF_q)$ is lossless. Similarly, if $q\equiv 1\mod 4$ then the Sylow $2$-subgroup of $\mathsf{PSU}_3(\FF_q)$ is lossless.
	\end{corollary}
	\begin{proof}
		In each of these cases the Sylow $2$-subgroup is semidihedral~\cite{semidihedral}.
	\end{proof}
	
	\begin{example}
		Let $q\mid \phi(p^3)/p = p(p-1)$, and let a generator $g\in T = C_{pq}$ act on $N = \Z/p^3$ via multiplication by some element in $\Z/p^3$ of order $qp$, and let $G = N\rtimes T$. Let $N_1 = pN$ and $N_2 = p^2N$. Let $K = T^q$, $H = KN_2$, and $L = TN_2$. Note that $N_1,N_2$ are characteristic in $N$ and hence normal in $G$, so $H$ and $L$ are indeed subgroups of $G$.
		
		A computation shows that $H$ is normal in $G$, but if $p\neq 2$ then $N_G(K) = N_G(L) = TN_1 < G$. Thus $G$ is not lossless, showing the assumption that $[G:N]$ is prime in \autoref{metap} is essential.
	\end{example}
	
	\begin{prop}\label{extraspec}
		If the derived subgroup of $G$ has prime order, then $G$ is lossless.
	\end{prop}
	\begin{proof}
		Let $K,{}^gK\leqslant H$. Since $G/G'$ is Abelian, any subgroup containing $G'$ is normal in $G$. Thus in particular $KG'$ is normal, and hence ${}^gK\leqslant KG'$. Since $|G'|=p$ we have $[KG':K]\leqslant p$, so $K$ is maximal in $KG'$. Thus if ${}^gK\neq K$ then $\lrangle{K,{}^gK} = KG'$ and hence $G'\leqslant KG'\leqslant H$. Thus $H$ is normal, and as such, there is nothing to check.
	\end{proof}
	
	\begin{defn}
	A $p$-group $G$ is said to be \emph{extraspecial} if its center $Z(G)$ is cyclic of order $p$ and the quotient $G/Z(G)$ is a non-trivial elementary abelian $p$-group. 
	\end{defn}
	
	\begin{corollary}\label{prop:extraspec}
		Any extraspecial group is lossless.
	\end{corollary}
	
	\begin{corollary}
		For any prime $p$ the subgroup of upper triangular matrices in $\mathsf{GL}_2(\FF_p)$ is lossless.
	\end{corollary}
	
	\begin{prop}\label{prop:cp2}
		If $\gcd(m,p)=1$ and $G\cong (C_p)^2\rtimes C_m$, then $G$ is lossless.
	\end{prop}
	\begin{proof}
		Let $K,{}^gK\leqslant H$. Again any subgroup containing $G'$ is normal in $G$. If $K\cap G'\neq 1$, then by the diamond identity $[KG':K] = [G':G'\cap KG']\leqslant p$ and the same argument as before applies. Thus we can assume $K\cap G'=1$, and hence $K,{}^gK$ are contained in some $p'$-Hall subgroups $S,S'$ of $H$. Thus we can find some $h\in H$ such that ${}^hS = S'$ and hence ${}^{h}K\leqslant S'$. But $S'$ is a cyclic group so $|{}^hK| = |{}^gK|$ implies ${}^hK = {}^gK$, and of course $h\in N_G(H)$ as required.
	\end{proof}
	
	\begin{corollary}
		For any prime $p$ the subgroup of upper triangular matrices in $\mathsf{SL}_2(\FF_{p^2})$ is lossless.
	\end{corollary}
	
	\begin{example}\label{lem:goodsl}
		If $p=2,3,$ or $5$, then $\SL_2(\FF_p)$ is lossless.
	\end{example}
	
	In fact, for $p=2,3,5$ the group $\SL_2(\FF_p)$ satisfies a very strong additional property: any two isomorphic subgroups are conjugate. We shall say that such a group is \emph{universally lossless}. Groups like this are quite useful for identifying ``lossless pieces'' of larger groups containing them as subgroups. 
	
	\begin{prop}
		Let $G$ be an arbitrary (finite) group, and suppose $K,{}^gK\leqslant H\leqslant G$. Suppose further that $H\leqslant L\leqslant G$, where $L$ is a universally lossless group. Then $K$ and ${}^gK$ are conjugate in $N_G(H)$.
	\end{prop}
	\begin{proof}
		Since $K\cong {}^gK$ and $K,{}^gK\leqslant L$, by assumption $K$ and ${}^gK$ are conjugate in $L$, so we can assume $g\in L$. Thus since $L$ is lossless and $K,{}^gK\leqslant H\leqslant L$, we must have $K$ and ${}^gK$ are conjugate in $N_L(H)\leqslant N_G(H)$.
	\end{proof}
	
	\section{Lifting criteria}\label{sec:metacyclic}
	
	In \autoref{cor:losslesssummary} we saw that for a lossless group $G$, $G$-transfer systems are in bijection with liftable transfer systems on $\Sub(G)/G$. However as discussed in \autoref{rem:whatarelifts}, explicitly identifying which transfer systems on $\Sub(G)/G$ are liftable can be extremely difficult in general. In this section we consider some special cases where the lifting conditions are tangible. 	We first begin with a generality.
	
	\begin{lemma}\label{lem:lifting}
		Let $G$ be a lossless group, and let $\RR$ be a categorical transfer system on $\Sub(G)/G$. Then $\RR$ is liftable if and only if for all $[K]\, \RR\, [H]$ and any $K'\leqslant H$ with $[K'] = [K]$ we have $[K\cap K']\,\RR\,[H]$.
	\end{lemma}
	\begin{proof}
		Suppose $\RR$ is liftable and $K,K'\leqslant H$ are as in the statement of the lemma. Then $K\to H,K'\to H\in\pi^*(\RR)$ by definition, so by restriction-closure and transitivity we have $(K\cap K')\to K'\to H\in\pi^*(\RR)$ and hence $[K\cap K']\to [H]\in \pi_*(\pi^*(\RR)) = \RR$.
		
		Now suppose conversely that $\RR$ is a categorical transfer system on $\Sub(G)/G$ such that the condition in the statement of the lemma holds. We claim that $\pi^{-1}(\RR)$ is a $G$-transfer system. Note that if this claim is true then
		\[\pi_*(\pi^*(\RR)) = \pi_*(\pi^{-1}(\RR))\leqslant \RR\leqslant \pi_*(\pi^*(\RR))\]
		where the first inequality follows directly from the definition of $\pi_*$ and $\pi^{-1}$ and the last inequality follows from $\pi^*\dashv\pi_*$.

%		"...is that it is restriciton-closed. In other words, we need to show that given K \to H \in \pi^{-1}(\RR) and L \leqslant H we have K \cap L \to L \in \pi^{-1}(\RR). We prove this by induction on the tuple (K,H,L). In other words, assuming the claim is true for all tuples (K',H',L') with K' \leqslant K, L' \leqslant L, H' \leqslant H, and at least one of these inequalities is strict, we want to show this implies the claim for (K,H,L). The base case where K = H = L is the trivial subgroup holds by reflexivity of \RR".

		Clearly $\pi^{-1}(\RR)$ is a conjugation-closed partial order refining $\leqslant$. All we need to show is that it is restriction-closed. In other words, we need to show that given $K \to H \in \pi^{-1}(\RR)$ and $L \leqslant H$ we have $(K {\cap} L) \to L \in \pi^{-1}(\RR)$. We prove this by induction on the tuple $(K,H,L)$. That is, assuming the claim is true for all tuples $(K',H',L')$ with $K' \leqslant K, L' \leqslant L, H' \leqslant H$, and at least one of these inequalities is strict, we want to show this implies the claim for $(K,H,L)$. The base case where $K = H = L$ is the trivial subgroup holds by reflexivity of $\RR$.

		%So suppose $K\to H\in\pi^{-1}(\RR)$ and $L\leqslant H$. We can assume by induction that whenever $K'\leqslant K$, $H'\leqslant H$, $L'\leqslant L$, and at least one of these inequalities is strict, then $K'\to H'\in\pi^{-1}(\RR)$ and $L'\leqslant H'$ implies $K'\cap L'\to L'\in\pi^{-1}(\RR)$.
		
		By the induction hypothesis with $K$ and $H$ fixed but $L'<L$, we can assume $(K\cap L)\to L'\in\pi^{-1}(\RR)$ for all $L'<L$ with $K\cap L\leqslant L'$. Let $M\leqslant L$ such that $[K\cap L]\leqslant [M]$ and $[M]$ is maximal among $[M]\leqslant [K],[L]$. Let $g\in G$ such that ${}^g(K\cap L)\leqslant M$. Since $G$ is lossless we can assume $g\in N_G(L)$. Thus ${}^{g^{-1}}M\leqslant L$ and $[M] = [{}^{g^{-1}}M]$, so we can assume without loss of generality that $K\cap L\leqslant M$.
		
		If $M < L$, then by the induction hypothesis we have $K\cap L\to M\in\pi^{-1}(\RR)$ and hence by definition $[K\cap L]\,\RR\,[M]$. But since $\RR$ is a categorical transfer system, $[M]$ is maximal for $[M]\leqslant [K],[L]\leqslant [H]$, and $[K]\,\RR\,[H]$, we have $[K\cap L]\,\RR\,[M]\,\RR\,[L]$, and hence $K\cap L\to L\in \pi^{-1}(\RR)$.
		
		On the other hand, if $M = L$ then $[L]\leqslant [K]$, so ${}^{g^{-1}}L\leqslant K$ for some $g\in G$, and again since $G$ is lossless we can assume $g\in N_G(H)$. Thus $K,{}^gK\leqslant H$ and $[K]\,\RR\,[H]$, so by hypothesis we have $[K\cap {}^gK]\,\RR\,[H]$. By restriction-closure this then implies $[K\cap {}^gK]\,\RR\,[{}^gK]$, so $K\cap {}^gK\to {}^gK\in\pi^{-1}(\RR)$. If ${}^gK<H$, then by the induction hypothesis with $K' = K\cap {}^gK$, $H' = {}^gK$, and $L' = L$ we obtain $K\cap {}^gK\cap L = K\cap L\to L\in\pi^{-1}(\RR)$. But if ${}^gK = H$ then also $K = H$ and hence $K\cap L = L$ so there's nothing to show in this case.
	\end{proof}
	
	Although \autoref{lem:lifting} is nice in its generality, and is certainly more explicit than the basic definition, checking it still requires understanding subtle details about the way that subgroups of $G$ embed into each other. Thankfully, for certain groups the subgroup structure is nice enough to make this condition particularly explicit. Some of the results described below hold in greater generality than stated, but the purpose of this section is mainly illustrative so we avoid excess generality. We recall that, in essence, a \emph{Frobenius group} is a transitive permutation group on a finite set such that no non-trivial element fixes more than one point, and some non-trivial element fixes a point. They can be characterized as those groups $G$ possessing a proper, nontrivial subgroup $T$ (called the \emph{Frobenius complement}) such that $T \cap {}^gT$ is the trivial subgroup for every $g \in G \smallsetminus T$. The identity element along with members of $G\smallsetminus \bigcup_{g\in G}{}^g T$ form the \emph{Frobenius kernel} of $G$.

	\begin{defn}\label{defn:mfgroup}
		A \emph{metacyclic Frobenius group} (mcF group) is a Frobenius group $G$ such that both the kernel $N\trianglelefteqslant G$ and the complement $T\leqslant G$ are cyclic groups.
	\end{defn}
	
	\begin{remark}
		Any mcF group is of the form
		\[G\cong \Z/n\rtimes T\]
		where $T$ is a cyclic subgroup of $(\Z/n)^\times$ such that $x-1\in (\Z/n)^{\times}$ for all $x\neq 1\in T$. Conversely every group of this form is an mcF group. For a general mcF group $G$, we will assume that we have passed through this bijection, and for clarity, write shall write $N$ for the group $(\mathbb{Z}/n)$. That is, $G \cong N \rtimes T$.
	\end{remark}
	
	\begin{remark}
		Although not immediately obvious from the definition, one can show that a group $G$ is metacyclic Frobenius if and only if $G$ is both a metacyclic group and a Frobenius group, explaining the naming choice. This collection of groups also implicitly appears in work of Khukhro--Makareno \cite{MFgroups1, MFgroups2}, but the authors are not aware of any other place that they have been studied.
			\end{remark}
	
	\begin{example}\label{ex:dihedralmcf}
		Let $n$ be odd. Then the dihedral group $D_n = \mathbb{Z}/n \rtimes \mathbb{Z}/2$ of order $2n$ is a mcF group.
	\end{example}
	
	\begin{example}
		Let $p$ be any prime. Then $\mathsf{AGL}_1(\mathbb{F}_p) = \mathbb{F}_p\rtimes \mathbb{F}_p^\times$, the  group of affine linear transformations of the finite field $\mathbb{F}_p$, is a mcF group. Here we are using Gauss's observation that $(\mathbb{Z}/n)^\times$ is cyclic when $n = p^k$.
	\end{example}
	
	\begin{defn}
		Let $K\leqslant G$. We call $K\cap N$ the \emph{base} of $K$ and write $K\cap N = N_K$.
	\end{defn}
	
	\begin{lemma}\label{lem:mcfsub}
		Any subgroup $K$ of an mcF group $G$ such that $N_K\neq \{e\}$ and $K\not\leqslant N$ is itself an mcF group with kernel $N_K$. If $N_K = \{e\}$ or $K\leqslant N$ then $K$ is a cyclic group.
	\end{lemma}
	\begin{proof}
		This follows immediately from~\cite[Lemma 2.2]{feit}.
	\end{proof}
	
	\begin{lemma}\label{lem:mcfnormal}
		A subgroup $K$ of an mcF group $G$ is normal if and only if $K\leqslant N$ or $N_K = N$.
	\end{lemma}
	\begin{proof}
		If $N_K = N$ then $K = KN$ is the preimage of $KN/N\leqslant G/N$, and since $G/N\cong T$ is cyclic this implies $KN/N\trianglelefteqslant G/N$ and hence $K\trianglelefteqslant G$. If $K\leqslant N$ then $K$ is characteristic in $N\trianglelefteqslant G$ and hence again $K\trianglelefteqslant G$.
		
		Conversely suppose $N_K\neq N$ and $K\not\leqslant N$. Let $g\in K\setminus N$ and $h\in N\setminus K$. Since $N$ is cyclic and $\operatorname{ord}(h) = \operatorname{ord}(ghg^{-1})$, we can write $ghg^{-1} = h^m$, and since $G$ is Frobenius and $gN\neq eN\in G/N$, $m-1$ must be a unit mod $\operatorname{ord}(h)$. But if $K$ is normal then $h^{m-1} = h^{-1}h^m = (h^{-1}gh)g^{-1}\in K$ and hence $h\in K$, a contradiction.
	\end{proof}
	
	\begin{lemma}\label{lem:mcflossless}
		Any mcF group $G$ is in particular a solvable T-group, and hence lossless by \autoref{prop:tgroup}. Furthermore, any two subgroups with the same order are conjugate in $G$.
	\end{lemma}
	\begin{proof}
		Solvability follows from the fact that $G = N\rtimes T$ and $N,T$ are cyclic by definition. The fact that $G$ is a T-group follows directly from \autoref{lem:mcfsub} and \autoref{lem:mcfnormal}.
		
		Now suppose $K,K'\leqslant G$ and $|K| = |K'|$. Since $N$ is a normal Hall subgroup of $G$, we also have $|N_K| = |N_{K'}|$ and hence $N_K = N_{K'}$ since $N$ is cyclic. Let $S,S'$ be complementary Hall subgroups of $K$ and $K'$, respectively. Then we can extend $S,S'$ to Hall subgroups $T,T'$ of $G$, and by Hall's theorem we can find $g\in G$ such that ${}^gT = T'$. Then ${}^gS,S'\leqslant T'$ are subgroups of the cyclic group $T'$ and $|{}^gS| = |S'|$, so ${}^gS = S'$. Thus ${}^gK = {}^g(N_KS) = N_KS' = K'$.
	\end{proof}
	
	The main property that makes the lifting conditions for mcF groups simple comes from the following proposition.
	
	\begin{prop}\label{prop:intersection}
		Let $G$ be an mcF group, and let $K\leqslant G$. For all $g\in G\setminus N_G(K)$, we have $K\cap {}^gK = N_K$.
	\end{prop}
	\begin{proof}
		Since $N_K$ is a normal subgroup of $K$, we have $N_K\leqslant K\cap {}^gK$ for all $g\in G$. To complete the proof, we show that the existence of $x\in (K\cap {}^{g^{-1}}K)\setminus N_K$ implies that $g\in N_G(K)$. Fix such an $x$. Then $\operatorname{ord}(x)$ cannot divide $|N_K|$, so after replacing $x$ with some power we can assume $\operatorname{ord}(x)$ is coprime to $|N_K|$. Let $S\leqslant K$ be a Hall subgroup complementary to $N_K$ such that $x\in S$, and let $T\leqslant K$ be a Hall subgroup complementary to $N_K$ such that $gxg^{-1}\in T$. Let $S'\geqslant S$ and $T'\geqslant T$ be Hall subgroups of $G$ complementary to $N$. Then by Hall's theorem we can find $h\in G$ such that ${}^{hg}S' = T'$, and since $e\neq gxg^{-1}\in {}^gS'\cap T'$ this implies ${}^{hg}S'\cap {}^gS'\neq\{e\}$. By \cite{feit}, this implies we must have $T' = {}^{hg}S' = {}^gS'$. Since $T'$ is cyclic it has a unique subgroup of order $|T| = |{}^gS|$, so this then also implies $T = {}^gS$. But then $g\in N_G(K)$ as
		\[{}^gK = {}^g(N_KS) = N_KT = K.\]
	\end{proof}
	
	\begin{corollary}\label{cor:mcflifting}
		Let $G$ be an mcF group, and $\RR$ a categorical transfer system on $\Sub(G)/G$. Then $\RR$ is liftable if and only if whenever $[K]\,\RR\,[H]$ with $N_K\neq N_H$, we have $[N_K]\,\RR\,[H]$ (or equivalently $[N_K]\,\RR\,[K]$).
	\end{corollary}
	\begin{proof}
		By \autoref{lem:lifting}, for the if direction we need to check for all $K,{}^gK\leqslant H$ that either $K = {}^gK$ or $K\cap {}^gK = N_K$, but this follows directly from \autoref{prop:intersection}. For the only if direction we need to check that $N_K\neq N_H$ implies there exists some $g\in G$ with ${}^gK\leqslant H$ such that $N_K = K\cap {}^gK$. But $N_K\neq N_H$ implies by \autoref{lem:mcfnormal} that $K$ is not normal in $H$, and hence by \autoref{prop:intersection} we can find some $g\in H$ such that $K\cap {}^gK = N_K$, and of course ${}^gK\leqslant {}^gH = H$.
	\end{proof}
	
	\begin{remark}\label{rem:subgrplattice}
	We can make this result more visually intuitive as follows. By \autoref{lem:mcflossless}, the map $[K] \mapsto (|N_K|,[K:N_K])$ is a poset isomorphism $\Sub(G)/G\cong D_N\times D_T$, where $D_N$ is the lattice of divisors of $|N|$ and $D_T$ is the lattice of divisors of $|G/N| = |T|$. Then \autoref{cor:mcflifting} says that a categorical transfer system $\RR$ on $D_N\times D_T$ is liftable if and only if whenever $(i,j)\, \RR \, (i',j')$ with $j'>j$, we must have $(i,1)\, \RR \,(i,j)$.
	\end{remark}
	
	\section{Examples of liftable transfer systems}\label{sec:examples}
	
	We will now apply the theory presented in this paper to two classes of mcF groups.
	
	\subsection{Dihedral groups of prime power order}
	
	We begin by considering groups of the form $D_{p^k} \cong \mathbb{Z}/p^k \rtimes \mathbb{Z}/2$ where $p$ is an odd prime; these are mcF groups by \aref{ex:dihedralmcf}. From \aref{rem:subgrplattice}, it follows that $\Sub(G)/G \cong [k] \times [1]$ and we shall consider elements of this lattice as pairs $(i,j)$ where $i \in [k]$ and $j \in [1]$. Here $[n]$ is the totally ordered finite set $\{0<1<\cdots n\}$.  It will be useful for us to display this lattice as the horizontal ladder as in \autoref{lad}.
	
	\begin{figure}[h]
		\centering
		\begin{tikzpicture}[scale=1.25]
		\node[circle,draw=black, fill=black, inner sep=0pt, minimum size=5pt] (00) at (0,0) {};
		\node[circle,draw=black, fill=black, inner sep=0pt, minimum size=5pt] (10) at (1,0) {};
		\node[circle,draw=none, fill=none, inner sep=0pt, minimum size=5pt] (20) at (2,0) {$\cdots$};
		\node[circle,draw=black, fill=black, inner sep=0pt, minimum size=5pt] (30) at (3,0) {};
		\node[circle,draw=black, fill=black, inner sep=0pt, minimum size=5pt] (40) at (4,0) {};
		\node[circle,draw=black, fill=black, inner sep=0pt, minimum size=5pt] (01) at (0,1) {};
		\node[circle,draw=black, fill=black, inner sep=0pt, minimum size=5pt] (11) at (1,1) {};
		\node[circle,draw=none, fill=none, inner sep=0pt, minimum size=5pt] (21) at (2,1) {$\cdots$};
		\node[circle,draw=black, fill=black, inner sep=0pt, minimum size=5pt] (31) at (3,1) {};
		\node[circle,draw=black, fill=black, inner sep=0pt, minimum size=5pt] (41) at (4,1) {};
		\draw[thin,black!30,->] (00) -- (10);
		\draw[thin,black!30,->] (10) -- (20);
		\draw[thin,black!30,->] (20) -- (30);
		\draw[thin,black!30,->] (30) -- (40);
		\draw[thin,black!30,->] (01) -- (11);
		\draw[thin,black!30,->] (11) -- (21);
		\draw[thin,black!30,->] (21) -- (31);
		\draw[thin,black!30,->] (31) -- (41);
		\draw[thin,black!30,->] (00) -- (01);
		\draw[thin,black!30,->] (10) -- (11);
		\draw[thin,black!30,->] (30) -- (31);
		\draw[thin,black!30,->] (40) -- (41);
		\draw[decoration={brace,mirror,raise=5pt},decorate]   (0,0) -- node[below=6pt] {$k+1$} (4,0);	
		\end{tikzpicture}
		\caption{The lattice $[k] \times [1]$.}\label{lad}
	\end{figure}
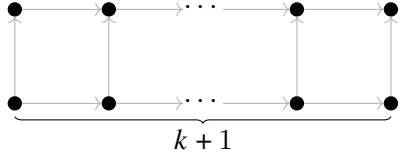
	
	\autoref{lad2} shows the corresponding subgroups (where we have used square brackets to denote conjugacy classes where required).
	
	\begin{figure}[h]
		\centering
		\begin{tikzpicture}[scale=1.25]
		\node[draw=none, fill=none, inner sep=0pt, minimum size=5pt] (00) at (0,0) {$e$};
		\node[draw=none, fill=none, inner sep=0pt, minimum size=5pt] (10) at (1,0) {$C_p$};
		\node[draw=none, fill=none, inner sep=0pt, minimum size=5pt] (20) at (2,0) {$\cdots$};
		\node[draw=none, fill=none, inner sep=0pt, minimum size=5pt] (30) at (3,0) {$C_{p^{k-1}}$};
		\node[draw=none, fill=none, inner sep=0pt, minimum size=5pt] (40) at (4,0) {$C_{p^k}$};
		\node[draw=none, fill=none, inner sep=0pt, minimum size=5pt] (01) at (0,1) {$[D_1]$};
		\node[draw=none, fill=none, inner sep=0pt, minimum size=5pt] (11) at (1,1) {$[D_p]$};
		\node[draw=none, fill=none, inner sep=0pt, minimum size=5pt] (21) at (2,1) {$\cdots$};
		\node[draw=none, fill=none, inner sep=0pt, minimum size=5pt] (31) at (3,1) {$[D_{p^{k-1}}]$};
		\node[draw=none, fill=none, inner sep=0pt, minimum size=5pt] (41) at (4,1) {$D_{p^k}$};
		\draw[thin,black!30,->] (00) -- (10);
		\draw[thin,black!30,->] (10) -- (20);
		\draw[thin,black!30,->] (20) -- (30);
		\draw[thin,black!30,->] (30) -- (40);
		\draw[thin,black!30,->] (01) -- (11);
		\draw[thin,black!30,->] (11) -- (21);
		\draw[thin,black!30,->] (21) -- (31);
		\draw[thin,black!30,->] (31) -- (41);
		\draw[thin,black!30,->] (00) -- (01);
		\draw[thin,black!30,->] (10) -- (11);
		\draw[thin,black!30,->] (30) -- (31);
		\draw[thin,black!30,->] (40) -- (41);
		\end{tikzpicture}
		\caption{The subgroups in the lattice $\Sub(D_{p^k})/D_{p^k}$.}\label{lad2}
	\end{figure}
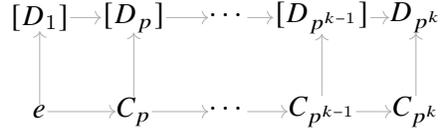

	We can now unravel \autoref{cor:mcflifting} in this specific example. We need to consider situations where we have $[K] \, \RR \, [H]$ with $N_K \neq N_H$. This occurs when we move horizontally on the top row of \autoref{lad}. In this case we require $[N_K] \, \RR \, [K]$ for it to lift to a transfer system for $D_{p^k}$ itself. All in all we conclude that the conditions of \autoref{cor:mcflifting} correspond to the following in terms of the group:
	
	\vspace{2mm}
	
	\begin{itemize}
		\item If $[D_{p^i}] \, \RR\,  [D_{p^j}]$ then $C_{p^i} \, \RR \, [D_{p^i}]$ for all $0 \leqslant i < j \leqslant k$.
	\end{itemize}
	
	\vspace{2mm}
	
	The following corollary rewords this condition in terms of the categorical transfer systems on $[k] \times [1]$ using \autoref{rem:subgrplattice}.
	
	\begin{corollary}\label{cor:dpk}
		A $D_{p^k}$-transfer system is equivalent to the data of a categorical transfer system on $[k] \times [1]$ which satisfies the following rule:
		\vspace{2mm}
	
		\begin{itemize}
			\item[$\star$] If $(i,1) \, \RR \, (i',1)$ for $i < i'$ then $(i,0) \, \RR \, (i,1)$ for all $0 \leqslant i < i' \leqslant k$.
		\end{itemize}
	\end{corollary}
	
	\begin{example}\label{ex:dp}
		Let us consider the case when $k=1$, so that the lattice in question is $[1] \times [1]$. Then condition $(\star)$ of \autoref{cor:dpk} boils down to the single implication of \aref{imp:1x1}:
		
		\begin{figure}[h]
			\centering
			\begin{tikzpicture}[scale=1.2]
			\node[circle,draw=black, fill=black, inner sep=0pt, minimum size=5pt] (00) at (0,0) {};
			\node[circle,draw=black, fill=black, inner sep=0pt, minimum size=5pt] (10) at (1,0) {};
			\node[circle,draw=black, fill=black, inner sep=0pt, minimum size=5pt] (01) at (0,1) {};
			\node[circle,draw=black, fill=black, inner sep=0pt, minimum size=5pt] (11) at (1,1) {};
			\draw[gRed,->] (01) -- (11);
			
			\node[circle,draw=none, fill=none, inner sep=0pt, minimum size=5pt] at (2,0.5) {$\Longrightarrow$};	
			\node[circle,draw=none, fill=none, inner sep=0pt, minimum size=5pt] at (2,1) {$(\star)$};
			
			\begin{scope}[xshift=80]
			\node[circle,draw=black, fill=black, inner sep=0pt, minimum size=5pt] (00') at (0,0) {};
			\node[circle,draw=black, fill=black, inner sep=0pt, minimum size=5pt] (10') at (1,0) {};
			\node[circle,draw=black, fill=black, inner sep=0pt, minimum size=5pt] (01') at (0,1) {};
			\node[circle,draw=black, fill=black, inner sep=0pt, minimum size=5pt] (11') at (1,1) {};
			\draw[gRed,->] (01') -- (11');
			\draw[gRed,->] (00') -- (01');
			\node[circle,draw=none, fill=none, inner sep=0pt, minimum size=5pt] at (1.3,-0.1) {.};
			\end{scope}
			\end{tikzpicture}
			\caption{The implications for $[1] \times [1]$.}\label{imp:1x1}
		\end{figure}
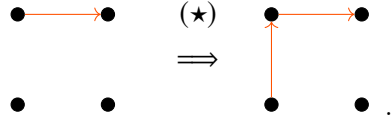
		
		% \newpage
		
		Note that this is exactly the observation of Rubin in~\cite{rubin} that we recalled in the introduction. Of the 10 transfer systems on $[1] \times [1]$, only one of them does not satisfy condition $(\star)$, namely:
		
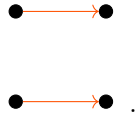
\begin{figure}[h]
			\centering
			\begin{tikzpicture}[scale=1.2]
			\node[circle,draw=black, fill=black, inner sep=0pt, minimum size=5pt] (00) at (0,0) {};
			\node[circle,draw=black, fill=black, inner sep=0pt, minimum size=5pt] (10) at (1,0) {};
			\node[circle,draw=black, fill=black, inner sep=0pt, minimum size=5pt] (01) at (0,1) {};
			\node[circle,draw=black, fill=black, inner sep=0pt, minimum size=5pt] (11) at (1,1) {};
			\draw[gRed,->] (01) -- (11);
			\draw[gRed,->] (00) -- (10);
			\node[circle,draw=none, fill=none, inner sep=0pt, minimum size=5pt] at (1.3,-0.1) {.};
			\end{tikzpicture}
		\end{figure}
		
	\end{example}
	
	\begin{example}\label{ex:dpp}
		We now move to the more exotic case of $k=2$ with conjugacy lattice $[2] \times [1]$. One can compute that there are 68 categorical transfer systems for this lattice. This time we have three possible options for the pair $i < i'$ in condition $(\star)$ which are given in \autoref{fig:2cond}.
		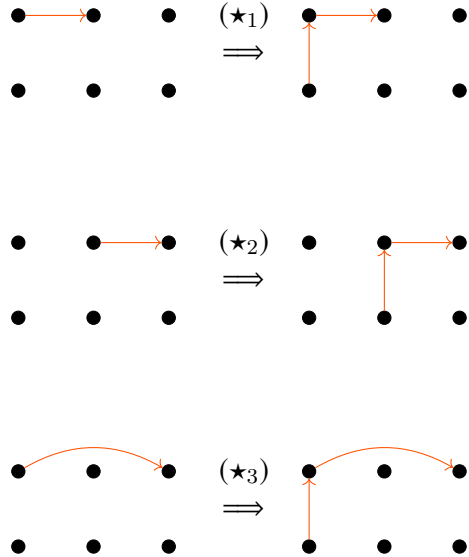
\begin{figure}[h]
			\centering
			\begin{tikzpicture}
			\node[circle,draw=black, fill=black, inner sep=0pt, minimum size=5pt] (00) at (0,0) {};
			\node[circle,draw=black, fill=black, inner sep=0pt, minimum size=5pt] (10) at (1,0) {};
			\node[circle,draw=black, fill=black, inner sep=0pt, minimum size=5pt] (01) at (0,1) {};
			\node[circle,draw=black, fill=black, inner sep=0pt, minimum size=5pt] (11) at (1,1) {};
			\node[circle,draw=black, fill=black, inner sep=0pt, minimum size=5pt] (20) at (2,0) {};
			\node[circle,draw=black, fill=black, inner sep=0pt, minimum size=5pt] (21) at (2,1) {};
			\draw[gRed,->] (01) -- (11);
			
			\node[circle,draw=none, fill=none, inner sep=0pt, minimum size=5pt] at (3,0.5) {$\Longrightarrow$};	
			\node[circle,draw=none, fill=none, inner sep=0pt, minimum size=5pt] at (3,1) {$(\star_1)$};
			
			\begin{scope}[xshift=110]
			\node[circle,draw=black, fill=black, inner sep=0pt, minimum size=5pt] (00') at (0,0) {};
			\node[circle,draw=black, fill=black, inner sep=0pt, minimum size=5pt] (10') at (1,0) {};
			\node[circle,draw=black, fill=black, inner sep=0pt, minimum size=5pt] (01') at (0,1) {};
			\node[circle,draw=black, fill=black, inner sep=0pt, minimum size=5pt] (11') at (1,1) {};
			\node[circle,draw=black, fill=black, inner sep=0pt, minimum size=5pt] (20') at (2,0) {};
			\node[circle,draw=black, fill=black, inner sep=0pt, minimum size=5pt] (21') at (2,1) {};
			\draw[gRed,->] (01') -- (11');
			\draw[gRed,->] (00') -- (01');
			\end{scope}
			\end{tikzpicture}
			
			\vspace{15mm}
			
			\begin{tikzpicture}
			\node[circle,draw=black, fill=black, inner sep=0pt, minimum size=5pt] (00) at (0,0) {};
			\node[circle,draw=black, fill=black, inner sep=0pt, minimum size=5pt] (10) at (1,0) {};
			\node[circle,draw=black, fill=black, inner sep=0pt, minimum size=5pt] (01) at (0,1) {};
			\node[circle,draw=black, fill=black, inner sep=0pt, minimum size=5pt] (11) at (1,1) {};
			\node[circle,draw=black, fill=black, inner sep=0pt, minimum size=5pt] (20) at (2,0) {};
			\node[circle,draw=black, fill=black, inner sep=0pt, minimum size=5pt] (21) at (2,1) {};
			\draw[gRed,->] (11) -- (21);
			
			\node[circle,draw=none, fill=none, inner sep=0pt, minimum size=5pt] at (3,0.5) {$\Longrightarrow$};	
			\node[circle,draw=none, fill=none, inner sep=0pt, minimum size=5pt] at (3,1) {$(\star_2)$};
			
			\begin{scope}[xshift=110]
			\node[circle,draw=black, fill=black, inner sep=0pt, minimum size=5pt] (00') at (0,0) {};
			\node[circle,draw=black, fill=black, inner sep=0pt, minimum size=5pt] (10') at (1,0) {};
			\node[circle,draw=black, fill=black, inner sep=0pt, minimum size=5pt] (01') at (0,1) {};
			\node[circle,draw=black, fill=black, inner sep=0pt, minimum size=5pt] (11') at (1,1) {};
			\node[circle,draw=black, fill=black, inner sep=0pt, minimum size=5pt] (20') at (2,0) {};
			\node[circle,draw=black, fill=black, inner sep=0pt, minimum size=5pt] (21') at (2,1) {};
			\draw[gRed,->] (11') -- (21');
			\draw[gRed,->] (10') -- (11');
			\end{scope}
			\end{tikzpicture}
			
			\vspace{15mm}
			
			\begin{tikzpicture}
			\node[circle,draw=black, fill=black, inner sep=0pt, minimum size=5pt] (00) at (0,0) {};
			\node[circle,draw=black, fill=black, inner sep=0pt, minimum size=5pt] (10) at (1,0) {};
			\node[circle,draw=black, fill=black, inner sep=0pt, minimum size=5pt] (01) at (0,1) {};
			\node[circle,draw=black, fill=black, inner sep=0pt, minimum size=5pt] (11) at (1,1) {};
			\node[circle,draw=black, fill=black, inner sep=0pt, minimum size=5pt] (20) at (2,0) {};
			\node[circle,draw=black, fill=black, inner sep=0pt, minimum size=5pt] (21) at (2,1) {};
			\draw[gRed,->] (01) to [bend left] (21);
			
			\node[circle,draw=none, fill=none, inner sep=0pt, minimum size=5pt] at (3,0.5) {$\Longrightarrow$};	
			\node[circle,draw=none, fill=none, inner sep=0pt, minimum size=5pt] at (3,1) {$(\star_3)$};
			
			\begin{scope}[xshift=110]
			\node[circle,draw=black, fill=black, inner sep=0pt, minimum size=5pt] (00') at (0,0) {};
			\node[circle,draw=black, fill=black, inner sep=0pt, minimum size=5pt] (10') at (1,0) {};
			\node[circle,draw=black, fill=black, inner sep=0pt, minimum size=5pt] (01') at (0,1) {};
			\node[circle,draw=black, fill=black, inner sep=0pt, minimum size=5pt] (11') at (1,1) {};
			\node[circle,draw=black, fill=black, inner sep=0pt, minimum size=5pt] (20') at (2,0) {};
			\node[circle,draw=black, fill=black, inner sep=0pt, minimum size=5pt] (21') at (2,1) {};
			\draw[gRed,->] (01') to [bend left] (21');
			\draw[gRed,->] (00') -- (01');
			\end{scope}
			\end{tikzpicture}
			\caption{The three restrictions needed for a transfer system on $D_{p^2}$.}\label{fig:2cond}
		\end{figure}
		
		Of the 68 transfer systems on $[2] \times [1]$, one can computationally verify that 56 of them are $D_{p^2}$-transfer systems.
	\end{example}
	
	In~\cite{bmodihedral}, the authors use these results to produce explicit recursion formul\ae{} for $D_{p^k}$-transfer systems.
	
	\newpage
	
	\subsection{Affine linear transformations of finite fields}
	
	In this section we will consider groups of the form $\mathsf{AGL}_1(\mathbb{F}_p) = \mathbb{F}_p \rtimes \mathbb{F}_p^\times$ where $p$ is prime. From \aref{rem:subgrplattice}, it follows that $\Sub(G)/G \cong [1] \times \Sub(\mathbb{Z}/(p-1))$. Of course, the prime factorization of $p-1$ follows no apparent rhyme or reason. For the convenience of the reader we list the first few values in \autoref{tab:sublat}.

	\begin{table}[h]
		\begin{tabular}{|c|c|}
			\hline
			$p$ & $\Sub(\mathbb{F}_p^\times)$ \\ \hline
			2   & $[0]$                       \\ \hline
			3   & $[1]$                       \\ \hline
			5   & $[2]$                       \\ \hline
			7   & $[1]\times [1]$             \\ \hline
			11  & $[1] \times [1]$            \\ \hline
			13  & $[2] \times [1]$            \\ \hline
			17  & $[4]$            \\ \hline
		\end{tabular}
		\vspace{5mm}
		\caption{The subgroup lattice of $\mathbb{F}_p^\times$.}\label{tab:sublat}
	\end{table}
	
	\begin{example}
		Consider $G=\mathsf{AGL}_1(\mathbb{F}_3)$. Here, the subgroup lattice  is $[1] \times [1]$. Then we are in the exact same case as \aref{ex:dp}, which is reassuring as $\mathsf{AGL}_1(\mathbb{F}_3) \cong D_{3}$. In particular there are 9 transfer systems for $G=\mathsf{AGL}_1(\mathbb{F}_3)$.
	\end{example}
		
	\begin{example}
		The first non-trivial example is $G=\mathsf{AGL}_1(\mathbb{F}_5)$, whose subgroup lattice is $[1] \times [2]$. We warn the reader that this case is not the same as \autoref{ex:dpp}. Indeed, even though $[1] \times [2] \cong [2] \times [1]$, the condition $(\star)$ is not invariant under this.  The lattice $\Sub(G)/G$ is depicted in \autoref{fig:agl1f5}.
		
		\vspace{5mm}
		
		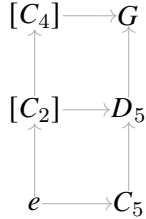
\begin{figure}[h]
			\centering
			\begin{tikzpicture}[scale=1.25]
			\node[draw=none, fill=none, inner sep=0pt, minimum size=5pt] (00) at (0,0) {$e$};
			\node[draw=none, fill=none, inner sep=0pt, minimum size=5pt] (10) at (1,0) {$C_5$};
			\node[draw=none, fill=none, inner sep=0pt, minimum size=5pt] (01) at (0,1) {$[C_2]$};
			\node[draw=none, fill=none, inner sep=0pt, minimum size=5pt] (11) at (1,1) {$D_5$};
			\node[draw=none, fill=none, inner sep=0pt, minimum size=5pt] (02) at (0,2) {$[C_4]$};
			\node[draw=none, fill=none, inner sep=0pt, minimum size=5pt] (12) at (1,2) {$G$};
			\draw[thin,black!30,->] (00) -- (10);
			\draw[thin,black!30,->] (01) -- (11);
			\draw[thin,black!30,->] (02) -- (12);
			\draw[thin,black!30,->] (00) -- (01);
			\draw[thin,black!30,->] (01) -- (02);
			\draw[thin,black!30,->] (10) -- (11);
			\draw[thin,black!30,->] (11) -- (12);
			\end{tikzpicture}
			\caption{The subgroups in the lattice $\Sub(\mathsf{AGL}_1(\mathbb{F}_5))/\mathsf{AGL}_1(\mathbb{F}_5)$.}\label{fig:agl1f5}
		\end{figure}
\vspace{5mm}

		Applying \autoref{cor:mcflifting} in conjunction with \aref{rem:subgrplattice} we obtain our lifting conditions that we collect in \aref{imp:2x1}.
		
		\newpage

		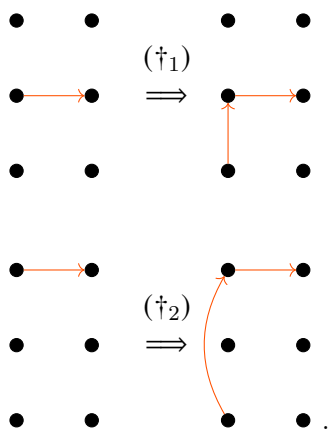
\begin{figure}[h]
			\centering
			\begin{tikzpicture}
			\node[circle,draw=black, fill=black, inner sep=0pt, minimum size=5pt] (00) at (0,0) {};
			\node[circle,draw=black, fill=black, inner sep=0pt, minimum size=5pt] (10) at (1,0) {};
			\node[circle,draw=black, fill=black, inner sep=0pt, minimum size=5pt] (01) at (0,1) {};
			\node[circle,draw=black, fill=black, inner sep=0pt, minimum size=5pt] (11) at (1,1) {};
			\node[circle,draw=black, fill=black, inner sep=0pt, minimum size=5pt] (02) at (0,2) {};
			\node[circle,draw=black, fill=black, inner sep=0pt, minimum size=5pt] (12) at (1,2) {};
			\draw[gRed,->] (01) -- (11);
			
			\node[circle,draw=none, fill=none, inner sep=0pt, minimum size=5pt] at (2,1) {$\Longrightarrow$};	
			\node[circle,draw=none, fill=none, inner sep=0pt, minimum size=5pt] at (2,1.5) {$(\dag_1)$};
			
			\begin{scope}[xshift=80]
			\node[circle,draw=black, fill=black, inner sep=0pt, minimum size=5pt] (00') at (0,0) {};
			\node[circle,draw=black, fill=black, inner sep=0pt, minimum size=5pt] (10') at (1,0) {};
			\node[circle,draw=black, fill=black, inner sep=0pt, minimum size=5pt] (01') at (0,1) {};
			\node[circle,draw=black, fill=black, inner sep=0pt, minimum size=5pt] (11') at (1,1) {};
			\node[circle,draw=black, fill=black, inner sep=0pt, minimum size=5pt] (02') at (0,2) {};
			\node[circle,draw=black, fill=black, inner sep=0pt, minimum size=5pt] (12') at (1,2) {};
			\draw[gRed,->] (01') -- (11');
			\draw[gRed,->] (00') -- (01');
			\node[circle,draw=none, fill=none, inner sep=0pt, minimum size=5pt] at (1.3,-0.1) {\textcolor{white}{.}};
			\end{scope}
			\end{tikzpicture}
			
			\vspace{10mm}
			
			\begin{tikzpicture}
			\node[circle,draw=black, fill=black, inner sep=0pt, minimum size=5pt] (00) at (0,0) {};
			\node[circle,draw=black, fill=black, inner sep=0pt, minimum size=5pt] (10) at (1,0) {};
			\node[circle,draw=black, fill=black, inner sep=0pt, minimum size=5pt] (01) at (0,1) {};
			\node[circle,draw=black, fill=black, inner sep=0pt, minimum size=5pt] (11) at (1,1) {};
			\node[circle,draw=black, fill=black, inner sep=0pt, minimum size=5pt] (02) at (0,2) {};
			\node[circle,draw=black, fill=black, inner sep=0pt, minimum size=5pt] (12) at (1,2) {};
			\draw[gRed,->] (02) -- (12);
			
			\node[circle,draw=none, fill=none, inner sep=0pt, minimum size=5pt] at (2,1) {$\Longrightarrow$};	
			\node[circle,draw=none, fill=none, inner sep=0pt, minimum size=5pt] at (2,1.5) {$(\dag_2)$};
			
			\begin{scope}[xshift=80]
			\node[circle,draw=black, fill=black, inner sep=0pt, minimum size=5pt] (00') at (0,0) {};
			\node[circle,draw=black, fill=black, inner sep=0pt, minimum size=5pt] (10') at (1,0) {};
			\node[circle,draw=black, fill=black, inner sep=0pt, minimum size=5pt] (01') at (0,1) {};
			\node[circle,draw=black, fill=black, inner sep=0pt, minimum size=5pt] (11') at (1,1) {};
			\node[circle,draw=black, fill=black, inner sep=0pt, minimum size=5pt] (02') at (0,2) {};
			\node[circle,draw=black, fill=black, inner sep=0pt, minimum size=5pt] (12') at (1,2) {};
			\draw[gRed,->] (02') -- (12');
			\draw[gRed,->] (00') to [bend left] (02');
			\node[circle,draw=none, fill=none, inner sep=0pt, minimum size=5pt] at (1.3,-0.1) {.};
			\end{scope}
			\end{tikzpicture}
			\caption{The implications for $[1] \times [2]$.}\label{imp:2x1}
		\end{figure}
		
		Out of the 68 transfer systems on $[1] \times [2]$, 59 of them satisfy condition $(\dag_1)$ and $(\dag_2)$.

	\end{example}
	
	\begin{example}
		Our final example is $G = \mathsf{AGL}_1(\mathbb{F}_7)$. Here $\Sub(G)/G \cong [1] \times ([1] \times [1])$, displayed in \autoref{fig:cube}.
		\begin{figure}[h]
			\centering
			\begin{tikzpicture}[scale=1]
			\node[draw=none, fill=none, inner sep=0pt, minimum size=5pt] (00) at (0,0) {$e$};
			\node[draw=none, fill=none, inner sep=0pt, minimum size=5pt] (20) at (2,0) {$C_7$};
			\node[draw=none, fill=none, inner sep=0pt, minimum size=5pt] (02) at (0,2) {$[C_2]$};
			\node[draw=none, fill=none, inner sep=0pt, minimum size=5pt] (22) at (2,2) {$D_7$};
			\node[draw=none, fill=none, inner sep=0pt, minimum size=5pt] (11) at (1,1) {$[C_3]$};
			\node[draw=none, fill=none, inner sep=0pt, minimum size=5pt] (31) at (3,1) {$C_7 \rtimes C_3$};
			\node[draw=none, fill=none, inner sep=0pt, minimum size=5pt] (13) at (1,3) {$[C_6]$};
			\node[draw=none, fill=none, inner sep=0pt, minimum size=5pt] (33) at (3,3) {$G$};
			\draw[thin,black!30,->] (00) -- (20);
			\draw[thin,black!30,->] (02) -- (22);
			\draw[thin,black!30,->] (11) -- (31);
			\draw[thin,black!30,->] (13) -- (33);
			\draw[thin,black!30,->] (00) -- (02);
			\draw[thin,black!30,->] (20) -- (22);
			\draw[thin,black!30,->] (11) -- (13);
			\draw[thin,black!30,->] (31) -- (33);
			\draw[thin,black!30,->] (00) -- (11);
			\draw[thin,black!30,->] (20) -- (31);
			\draw[thin,black!30,->] (02) -- (13);
			\draw[thin,black!30,->] (22) -- (33);
			\end{tikzpicture}
			\caption{The subgroups in the lattice $\Sub(\mathsf{AGL}_1(\mathbb{F}_7))/\mathsf{AGL}_1(\mathbb{F}_7)$.}\label{fig:cube}
		\end{figure}
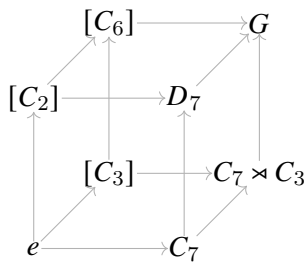
		
		As with the previous examples, we apply \autoref{cor:mcflifting} in conjunction with \aref{rem:subgrplattice} to obtain our lifting conditions as in \autoref{fig:duberel}.

		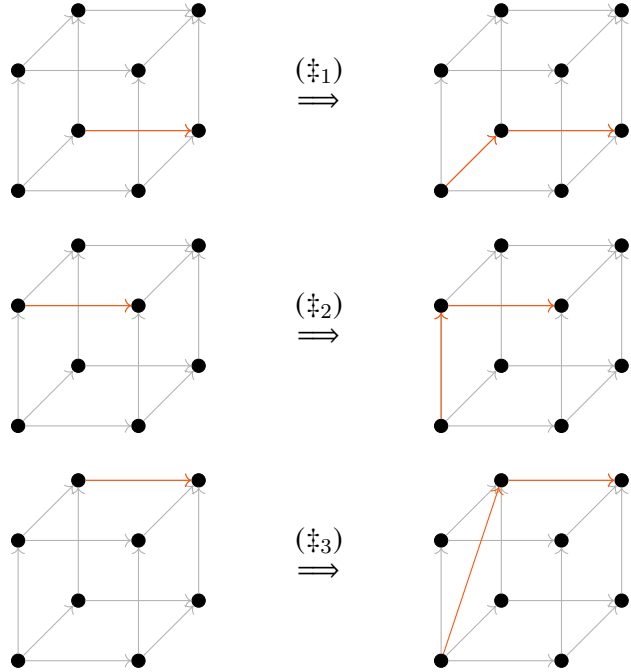
\begin{figure}[h]
			\centering
			\begin{tikzpicture}[scale=0.8]
			\node[circle,draw=black, fill=black, inner sep=0pt, minimum size=5pt] (00) at (0,0) {};
			\node[circle,draw=black, fill=black, inner sep=0pt, minimum size=5pt] (20) at (2,0) {};
			\node[circle,draw=black, fill=black, inner sep=0pt, minimum size=5pt] (02) at (0,2) {};
			\node[circle,draw=black, fill=black, inner sep=0pt, minimum size=5pt] (22) at (2,2) {};
			\node[circle,draw=black, fill=black, inner sep=0pt, minimum size=5pt] (11) at (1,1) {};
			\node[circle,draw=black, fill=black, inner sep=0pt, minimum size=5pt] (31) at (3,1) {};
			\node[circle,draw=black, fill=black, inner sep=0pt, minimum size=5pt] (13) at (1,3) {};
			\node[circle,draw=black, fill=black, inner sep=0pt, minimum size=5pt] (33) at (3,3) {};
			\draw[thin,black!30,->] (00) -- (20);
			\draw[thin,black!30,->] (02) -- (22);
			\draw[thin,black!30,->] (11) -- (31);
			\draw[thin,black!30,->] (13) -- (33);
			\draw[thin,black!30,->] (00) -- (02);
			\draw[thin,black!30,->] (20) -- (22);
			\draw[thin,black!30,->] (11) -- (13);
			\draw[thin,black!30,->] (31) -- (33);
			\draw[thin,black!30,->] (00) -- (11);
			\draw[thin,black!30,->] (20) -- (31);
			\draw[thin,black!30,->] (02) -- (13);
			\draw[thin,black!30,->] (22) -- (33);
			\draw[gRed,->] (11) -- (31);
			
			\node[circle,draw=none, fill=none, inner sep=0pt, minimum size=5pt] at (5,1.5) {$\Longrightarrow$};	
			\node[circle,draw=none, fill=none, inner sep=0pt, minimum size=5pt] at (5,2) {$(\ddag_1)$};
			
			\begin{scope}[xshift = 200]
			\node[circle,draw=black, fill=black, inner sep=0pt, minimum size=5pt] (00) at (0,0) {};
			\node[circle,draw=black, fill=black, inner sep=0pt, minimum size=5pt] (20) at (2,0) {};
			\node[circle,draw=black, fill=black, inner sep=0pt, minimum size=5pt] (02) at (0,2) {};
			\node[circle,draw=black, fill=black, inner sep=0pt, minimum size=5pt] (22) at (2,2) {};
			\node[circle,draw=black, fill=black, inner sep=0pt, minimum size=5pt] (11) at (1,1) {};
			\node[circle,draw=black, fill=black, inner sep=0pt, minimum size=5pt] (31) at (3,1) {};
			\node[circle,draw=black, fill=black, inner sep=0pt, minimum size=5pt] (13) at (1,3) {};
			\node[circle,draw=black, fill=black, inner sep=0pt, minimum size=5pt] (33) at (3,3) {};
			\draw[thin,black!30,->] (00) -- (20);
			\draw[thin,black!30,->] (02) -- (22);
			\draw[thin,black!30,->] (11) -- (31);
			\draw[thin,black!30,->] (13) -- (33);
			\draw[thin,black!30,->] (00) -- (02);
			\draw[thin,black!30,->] (20) -- (22);
			\draw[thin,black!30,->] (11) -- (13);
			\draw[thin,black!30,->] (31) -- (33);
			\draw[thin,black!30,->] (00) -- (11);
			\draw[thin,black!30,->] (20) -- (31);
			\draw[thin,black!30,->] (02) -- (13);
			\draw[thin,black!30,->] (22) -- (33);
			\draw[gRed,->] (11) -- (31);
			\draw[gRed,->] (00) -- (11);
			\end{scope}	
			\end{tikzpicture}	
			
			\vspace{5mm}
			
			\begin{tikzpicture}[scale=0.8]
			\node[circle,draw=black, fill=black, inner sep=0pt, minimum size=5pt] (00) at (0,0) {};
			\node[circle,draw=black, fill=black, inner sep=0pt, minimum size=5pt] (20) at (2,0) {};
			\node[circle,draw=black, fill=black, inner sep=0pt, minimum size=5pt] (02) at (0,2) {};
			\node[circle,draw=black, fill=black, inner sep=0pt, minimum size=5pt] (22) at (2,2) {};
			\node[circle,draw=black, fill=black, inner sep=0pt, minimum size=5pt] (11) at (1,1) {};
			\node[circle,draw=black, fill=black, inner sep=0pt, minimum size=5pt] (31) at (3,1) {};
			\node[circle,draw=black, fill=black, inner sep=0pt, minimum size=5pt] (13) at (1,3) {};
			\node[circle,draw=black, fill=black, inner sep=0pt, minimum size=5pt] (33) at (3,3) {};
			\draw[thin,black!30,->] (00) -- (20);
			\draw[thin,black!30,->] (02) -- (22);
			\draw[thin,black!30,->] (11) -- (31);
			\draw[thin,black!30,->] (13) -- (33);
			\draw[thin,black!30,->] (00) -- (02);
			\draw[thin,black!30,->] (20) -- (22);
			\draw[thin,black!30,->] (11) -- (13);
			\draw[thin,black!30,->] (31) -- (33);
			\draw[thin,black!30,->] (00) -- (11);
			\draw[thin,black!30,->] (20) -- (31);
			\draw[thin,black!30,->] (02) -- (13);
			\draw[thin,black!30,->] (22) -- (33);
			\draw[gRed,->] (02) -- (22);
			
			\node[circle,draw=none, fill=none, inner sep=0pt, minimum size=5pt] at (5,1.5) {$\Longrightarrow$};	
			\node[circle,draw=none, fill=none, inner sep=0pt, minimum size=5pt] at (5,2) {$(\ddag_2)$};
			
			\begin{scope}[xshift = 200]
			\node[circle,draw=black, fill=black, inner sep=0pt, minimum size=5pt] (00) at (0,0) {};
			\node[circle,draw=black, fill=black, inner sep=0pt, minimum size=5pt] (20) at (2,0) {};
			\node[circle,draw=black, fill=black, inner sep=0pt, minimum size=5pt] (02) at (0,2) {};
			\node[circle,draw=black, fill=black, inner sep=0pt, minimum size=5pt] (22) at (2,2) {};
			\node[circle,draw=black, fill=black, inner sep=0pt, minimum size=5pt] (11) at (1,1) {};
			\node[circle,draw=black, fill=black, inner sep=0pt, minimum size=5pt] (31) at (3,1) {};
			\node[circle,draw=black, fill=black, inner sep=0pt, minimum size=5pt] (13) at (1,3) {};
			\node[circle,draw=black, fill=black, inner sep=0pt, minimum size=5pt] (33) at (3,3) {};
			\draw[thin,black!30,->] (00) -- (20);
			\draw[thin,black!30,->] (02) -- (22);
			\draw[thin,black!30,->] (11) -- (31);
			\draw[thin,black!30,->] (13) -- (33);
			\draw[thin,black!30,->] (00) -- (02);
			\draw[thin,black!30,->] (20) -- (22);
			\draw[thin,black!30,->] (11) -- (13);
			\draw[thin,black!30,->] (31) -- (33);
			\draw[thin,black!30,->] (00) -- (11);
			\draw[thin,black!30,->] (20) -- (31);
			\draw[thin,black!30,->] (02) -- (13);
			\draw[thin,black!30,->] (22) -- (33);	
			\draw[gRed,->] (02) -- (22);
			\draw[gRed,->] (00) -- (02);
			\end{scope}	
			\end{tikzpicture}	
			
			\vspace{5mm}
			
			\begin{tikzpicture}[scale=0.8]
			\node[circle,draw=black, fill=black, inner sep=0pt, minimum size=5pt] (00) at (0,0) {};
			\node[circle,draw=black, fill=black, inner sep=0pt, minimum size=5pt] (20) at (2,0) {};
			\node[circle,draw=black, fill=black, inner sep=0pt, minimum size=5pt] (02) at (0,2) {};
			\node[circle,draw=black, fill=black, inner sep=0pt, minimum size=5pt] (22) at (2,2) {};
			\node[circle,draw=black, fill=black, inner sep=0pt, minimum size=5pt] (11) at (1,1) {};
			\node[circle,draw=black, fill=black, inner sep=0pt, minimum size=5pt] (31) at (3,1) {};
			\node[circle,draw=black, fill=black, inner sep=0pt, minimum size=5pt] (13) at (1,3) {};
			\node[circle,draw=black, fill=black, inner sep=0pt, minimum size=5pt] (33) at (3,3) {};
			\draw[thin,black!30,->] (00) -- (20);
			\draw[thin,black!30,->] (02) -- (22);
			\draw[thin,black!30,->] (11) -- (31);
			\draw[thin,black!30,->] (13) -- (33);
			\draw[thin,black!30,->] (00) -- (02);
			\draw[thin,black!30,->] (20) -- (22);
			\draw[thin,black!30,->] (11) -- (13);
			\draw[thin,black!30,->] (31) -- (33);
			\draw[thin,black!30,->] (00) -- (11);
			\draw[thin,black!30,->] (20) -- (31);
			\draw[thin,black!30,->] (02) -- (13);
			\draw[thin,black!30,->] (22) -- (33);
			\draw[gRed,->] (13) -- (33);
			
			\node[circle,draw=none, fill=none, inner sep=0pt, minimum size=5pt] at (5,1.5) {$\Longrightarrow$};	
			\node[circle,draw=none, fill=none, inner sep=0pt, minimum size=5pt] at (5,2) {$(\ddag_3)$};
			
			\begin{scope}[xshift = 200]
			\node[circle,draw=black, fill=black, inner sep=0pt, minimum size=5pt] (00) at (0,0) {};
			\node[circle,draw=black, fill=black, inner sep=0pt, minimum size=5pt] (20) at (2,0) {};
			\node[circle,draw=black, fill=black, inner sep=0pt, minimum size=5pt] (02) at (0,2) {};
			\node[circle,draw=black, fill=black, inner sep=0pt, minimum size=5pt] (22) at (2,2) {};
			\node[circle,draw=black, fill=black, inner sep=0pt, minimum size=5pt] (11) at (1,1) {};
			\node[circle,draw=black, fill=black, inner sep=0pt, minimum size=5pt] (31) at (3,1) {};
			\node[circle,draw=black, fill=black, inner sep=0pt, minimum size=5pt] (13) at (1,3) {};
			\node[circle,draw=black, fill=black, inner sep=0pt, minimum size=5pt] (33) at (3,3) {};
			\draw[thin,black!30,->] (00) -- (20);
			\draw[thin,black!30,->] (02) -- (22);
			\draw[thin,black!30,->] (11) -- (31);
			\draw[thin,black!30,->] (13) -- (33);
			\draw[thin,black!30,->] (00) -- (02);
			\draw[thin,black!30,->] (20) -- (22);
			\draw[thin,black!30,->] (11) -- (13);
			\draw[thin,black!30,->] (31) -- (33);
			\draw[thin,black!30,->] (00) -- (11);
			\draw[thin,black!30,->] (20) -- (31);
			\draw[thin,black!30,->] (02) -- (13);
			\draw[thin,black!30,->] (22) -- (33);
			\draw[gRed,->] (13) -- (33);
			\draw[gRed,->] (00) -- (13);
			\end{scope}	
			\end{tikzpicture}
			\caption{The three restrictions needed for a transfer system on $\mathsf{AGL}_1(\mathbb{F}_7)$.}\label{fig:duberel}
		\end{figure}
		
		Note again that all other possible relations are implied by these ones due to the usual axioms for a transfer system. 
		
		Out of the 450 transfer systems on $[1] \times [1] \times [1]$ (c.f.,~\cite{bbpr}), one computes that 400 of these are transfer systems for $\mathsf{AGL}_1(\mathbb{F}_7)$.
	\end{example}
	
	\newpage
	
	\section{Strategies for lossy groups}\label{sec:lossy}

	In the previous sections we have explored lossless groups, and shown that they provide a convenient computational framework for computing transfer systems, especially when restricted to nicer subclasses such as metacyclic Frobenius groups. Although we've shown in \autoref{sec:families} that several important classes of groups are lossless, more complicated groups that may arise in practice tend to be lossy. Although lossy groups prevent us from working directly with $\Sub(G)/G$, with some cleverness it may still be the case that we can obtain simpler representations of transfer systems for lossy groups. In this final section we speculatively discuss one possible strategy for dealing with lossy groups in the case of $G = \SL_2(\FF_p)$ where $p$ is prime.
	
	If $p=2,3,$ or $5$, then we have seen in \autoref{lem:goodsl} that $G$ is lossless, but for $p>5$ these groups are always lossy. On the other hand, if $p\equiv \pm 3\mod 8$ then these groups are very close to being lossless. When $p\equiv \pm 1\mod 8$, the lossy binary octahedral subgroups add to the lossyness, but even in this case $G$ only has a few deviations from losslessness. In the speculations that follow we focus on the simple case of $p\equiv \pm 3\mod 8$, but with some work it should be feasible to extend our constructions to work in general. In \autoref{fig:sl2f13} we provide a schematic for $\Sub(G)/G$ where $G=\mathsf{SL}_2(\FF_{13})$. The general goal is to represent $G$-transfer systems as a pair of abstract transfer systems on two small posets related to $\Sub(G)/G$, subject to a compatibility condition. For $p\equiv \pm1\mod 8$ one would likely need to use more than two posets, but the same general technique should still apply.

	Let $H_- \cong \Dic_{(p-1)/2}$ be the normalizer of the diagonal matrices, and $H_+\cong\Dic_{(p+1)/2}$ the normalizer of some maximal non-split torus. (Here $\Dic_n$ is the \emph{dicyclic group} of order $4n$ as in \aref{defn:groups}). Let $Z = H_1\cap H_2\cong C_2$ be the center of $G$. 
	
	Recall that for a group $G$, the \emph{Frattini subgroup} of $G$ is defined as intersection of all maximal subgroups of $G$~\cite[\S 5.1]{gorenstein}. The maximal subgroups of $\SL_2(\FF_p)$ can be deduced from \cite{king} since the Frattini subgroup of $\SL_2(\FF_p)$ is equal to its center. When $p\equiv \pm 3\mod 8$, these maximal subgroups are 
	\begin{enumerate}
		\item[(1)]  normalizers of torii, which are dicyclic as described above,
		\item[(2)]  Borel subgroups, which are isomorphic to $\FF_p\rtimes \FF_p^\times$ with $x\in\FF_p^\times$ acting on $\FF_p$ as multiplication by $x^2$, and
		\item[(3)]  binary tetrahedral or binary icosahedral subgroups.
	\end{enumerate}
	
	The Borel subgroups are universally lossless, and the binary tetrahedral/icosahedral subgroups are isomorphic to $\SL_2(\FF_3)$ and $\SL_2(\FF_5)$ which again are universally lossless. Furthermore, if we let $\epsilon = \pm 1$ accordingly as $p+\epsilon\equiv 4\mod 8$, then the normalizer of a torus of order $p-\epsilon$ is a dicyclic group of order $2(p-\epsilon)\equiv 4\mod 8$, and hence is also universally lossless. Thus the only non-universally lossless maximal subgroups are the normalizers of torii of order $p+\epsilon$, which are all conjugate to $H_{\epsilon}$.
	
	By \cite{cyclicconj}, any two cyclic subgroups of $G$ with the same order are conjugate. The subgroup $H_{\epsilon}$ contains three conjugacy classes of subgroups isomorphic to $C_4$, and hence when we embed into $\Sub(G)/G$ these three copies of $C_4$ must be mapped to the same conjugacy class. But if $C_4,{}^gC_4\leqslant K\leqslant H_{\epsilon}$ is not contained in any other maximal subgroup (e.g., $K = H_\epsilon$ itself), then $N_G(K)\leqslant H_{\epsilon}$, so $C_4,{}^gC_4$ cannot be conjugate in $N_G(K)$. This is what causes $G$ to be lossy. On the other hand, when $p\equiv \pm3\mod 8$ this appears to be the only obstacle preventing losslessness.
	
	Let $D_G$ be the poset $\Sub(H_\epsilon)/H_\epsilon$ with an additional top vertex $[G]$. (For $G=\SL_2(\FF_{13})$, this is depicted in \autoref{fig:sl2f13dg}.) Let $U_G\subseteq \Sub(G)/G$ be the subposet on objects $[H]$ such that either $H=G$ or $H$ is contained in some universally lossless subgroup (depicted in \autoref{fig:sl2f13}). We have natural poset maps $\psi^D\colon D_G\to \Sub(G)/G$ and $\psi^U\colon U_G\to \Sub(G)/G$. Let $I_G = (\psi^D)^{-1}(\im \psi^U)\subseteq D_G$ (depicted in \autoref{fig:sl2f13dg} for $G=\SL_2(\FF_{13})$). We let $\phi^D\colon I_G\to D_G$ be the canonical embedding, and we let $\phi^U\colon I_G\to U_G$ be the restriction of $\psi^D$. For any abstract transfer system $\RR$ on $I_G$, let $\phi^D_*(\RR) = \im(\phi^D|_{\RR})$ and similarly for $\phi^U_*(\RR)$.
	
	\begin{defn}
		A \emph{split transfer system} is a triple of catgorical transfer systems $\RR_D$, $\RR_I$, $\RR_U$ on $D_G$, $I_G$, $U_G$, respectively, such that
		\begin{enumerate}
			\item[(1)] if for some $[C_4]\in D_G$ we have $[C_4]\, \RR \,[G]$, then in fact $[C_4]\, \RR \, [G]$ for all conjugacy classes of $[C_4]\in D_G$, and
			\item[(2)] $\phi^D_*(\RR_I) = \RR_D\cap\im\phi^D$ and $\phi^U_*(\RR_I) = \RR_U\cap\im\phi^U$.
		\end{enumerate}
	\end{defn}
	
	For every split transfer system $(\RR_D,\RR_I,\RR_U)$, we can define a reflexive relation $\RR$ on $\Sub(G)$ as follows. Let $K\leqslant H$. If $H = G$ and $K\leqslant L$ for some universally lossless maximal subgroup $L$, then we set $K \, \RR \, H$ if and only if $[K] \, \RR_U \, [G]$. If $K\not\leqslant L$ for any such $L$, then some conjugate ${}^gK$ of $K$ is contained in $H_{\epsilon}$, and we set $K \, \RR \, H$ if and only if $[{}^gK] \, \RR_D \, [G]$.
	
	So we suppose $H<G$. If $H\leqslant L$ for some universally lossless maximal subgroup $L\leqslant G$, then we set $K \, \RR \, H$ if and only if $[K] \, \RR_U \, [H]$. Otherwise we can find some $g\in G$ such that ${}^gH\leqslant H_\epsilon$. Then we set $K \, \RR \, H$ if and only if $[{}^gK] \, \RR_D \, [{}^gH]$.
	
	Conversely, if $\RR$ is a transfer system on $G$ then we can define $(\RR_D,\RR_I,\RR_U)$ so that
	\begin{enumerate}
		\item[(1)] $\RR_D = \pi_*(\RR|_{\Sub(H_\epsilon)})$, where $\pi\colon \Sub(H_\epsilon)\to\Sub(H_\epsilon)/H_\epsilon\subseteq D_G$,
		\item[(2)] $\RR_I = \RR_D|_{I_G}$, and
		\item[(3)]  $\RR_U = \pi'_*(\RR)|_{U_G}$ where $\pi'\colon \Sub(G)\to\Sub(G)/G$.
	\end{enumerate} 
	
	The discussion of this speculative section culminates in the following conjecture. If this conjecture were true, it would provide a constructive method for exploring $N_\infty$ operads for an interesting class of groups which are not lossless. In particular, one should not despair if their favorite group of equivariance fails to be lossless, one only needs to figure out a way to exploit the structure of the group itself.
	
	\begin{conjecture}
		Fix an arbitrary transfer system $\RR$ on $G = \SL_2(\FF_p)$ where $p > 5$, $p \neq 11$, and $p \equiv \pm 3 \mod 8$.  Then the triple $(\RR_D,\RR_I,\RR_U)$ is a split transfer system, and $\RR$ is lifted from $(\RR_D,\RR_I,\RR_U)$ using the procedure above.
	\end{conjecture}
	
\begin{figure}[h]
\centering
	\begin{tikzpicture}[xscale=0.6, yscale = 0.55]
		\node[draw=none, gRed] (12) at (-10.25, -6) {$[e]$};
		\node [draw=none, gRed] (13) at (-6.5, -3) {$[C_3]$};
		\node [draw=none, gRed] (14) at (-13.25, -3) {$[C_2]$};
		\node [draw=none, gRed] (15) at (-6.5, 0) {$[C_6]$};
		\node [draw=none, gRed] (16) at (-13.25, 0) {$[C_4]$};
		\node [draw=none] (19) at (-8.75, 3) {$[\Dic_3]$};
		\node [draw=none] (20) at (-11, 3) {$[\Dic_3]$};
		\node [draw=none, gRed] (21) at (-6.5, 3) {$[C_{12}]$};
		\node [draw=none, gRed] (22) at (-13.25, 3) {$[Q_8]$};
		\node [draw=none] (23) at (-9.75, 6) {$[\Dic_6]$};
		\node [draw=none, gRed] (24) at (-15.75, 6) {$[\Dic_7]$};
		\node [draw=none, gRed] (25) at (-18, 0) {$[C_{14}]$};
		\node [draw=none, gRed] (26) at (-18, -3) {$[C_7]$};
		\node [draw=none, gRed] (27) at (-0.75, -3) {$[C_{13}]$};
		\node [draw=none, gRed] (28) at (-2.75, 0) {$[C_{26}]$};
		\node [draw=none, gRed] (29) at (1.5, 0) {$[C_{13} \rtimes C_{3}]$};
		\node [draw=none, gRed] (30) at (1.5, 3) {$[C_{2} \times C_{13} \rtimes C_{3}]$};
		\node [draw=none, gRed] (31) at (-2.75, 3) {$[\Dic_{13}]$};
		\node [draw=none, gRed] (32) at (-0.5, 6) {$[C_{13} \rtimes C_{12}]$};
		\node [draw=none, gRed] (33) at (-4.25, 6) {$[\SL_2(\mathbb{F}_{3})]$};
		\node [draw=none, gRed] (34) at (-9.5, 9) {$[\SL_2(\mathbb{F}_{13})]$};
		\draw[thin,black!30,->] (12) to (13);
		\draw[thin,black!30,->] (12) to (14);
		\draw[thin,black!30,->] (14) to (16);
		\draw[thin,black!30,->] (13) to (15);
		\draw[thin,black!30,->] (14) to (15);
		\draw[thin,black!30,->] (15) to (19);
		\draw[thin,black!30,->] (15) to (20);
		\draw[thin,black!30,->] (15) to (21);
		\draw[thin,black!30,->] (16) to (22);
		\draw[thin,black!30,->] (16) to (19);
		\draw[thin,black!30,->] (19) to (23);
		\draw[thin,black!30,->] (20) to (23);
		\draw[thin,black!30,->] (21) to (23);
		\draw[thin,black!30,->] (22) to (23);
		\draw[thin,black!30,->] (16) to (20);
		\draw[thin,black!30,->] (16) to (21);
		\draw[thin,black!30,->] (16) to (24);
		\draw[thin,black!30,->] (25) to (24);
		\draw[thin,black!30,<-] [in=90, out=-90] (25) to (26);
		\draw[thin,black!30,->] (12) to (26);
		\draw[thin,black!30,->] (14) to (25);
		\draw[thin,black!30,->] (12) to (27);
		\draw[thin,black!30,->] (27) to (28);
		\draw[thin,black!30,->] (14) to (28);
		\draw[thin,black!30,<-] (31) to (28);
		\draw[thin,black!30,<-] (31) to (16);
		\draw[thin,black!30,<-] (30) to (28);
		\draw[thin,black!30,<-] (30) to (15);
		\draw[thin,black!30,<-] (30) to (29);
		\draw[thin,black!30,->] (13) to (29);
		\draw[thin,black!30,->] (29) to (27);
		\draw[thin,black!30,<-] (32) to (30);
		\draw[thin,black!30,<-] (32) to (31);
		\draw[thin,black!30,->] (21) to (32);
		\draw[thin,black!30,->] (22) to (33);
		\draw[thin,black!30,->] (15) to (33);
		\draw[thin,black!30,->] (24) to (34);
		\draw[thin,black!30,<-] (34) to (23);
		\draw[thin,black!30,<-] (34) to (33);
		\draw[thin,black!30,<-] (34) to (32);
\end{tikzpicture}
\caption{The poset $\Sub(G)/G$ for $G = \SL_2(\mathbb{F}_{13})$. The non-split torus appears as $\Dic_7$, the split torus as $\Dic_6$, $C_{13} \rtimes C_{12}$ is the Borel subgroup and $\SL_2(\mathbb{F}_3)$ is the binary tetrahedral subgroup.  The objects of the sub-poset $U_G$ are highlighted in red.}\label{fig:sl2f13}
\end{figure}
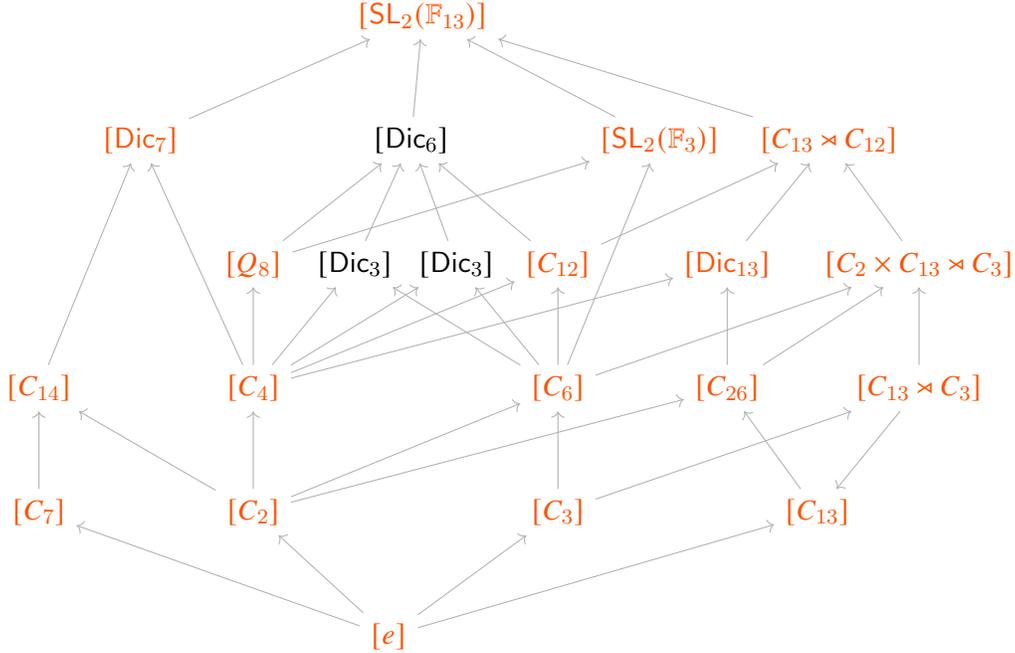

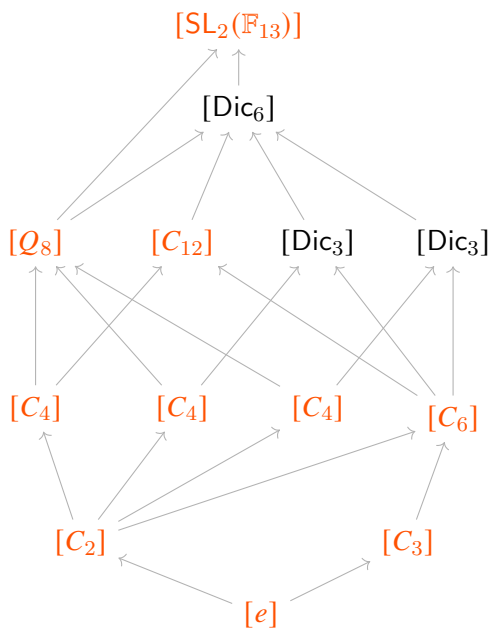
\begin{figure}[h]
\centering
	\begin{tikzpicture}[xscale=0.6, yscale = 0.55]
		\node[draw=none, gRed] (0) at (11.75, -5) {$[e]$};
		\node [draw=none,gRed] (1) at (15, -3.25) {$[C_3]$};
		\node [draw=none, gRed] (2) at (7.75, -3.25) {$[C_2]$};
		\node [draw=none,gRed] (3) at (16, -0.25) {$[C_6]$};
		\node [draw=none, gRed] (4) at (13, 0) {$[C_4]$};
		\node [draw=none, gRed] (5) at (10, 0) {$[C_4]$};
		\node [draw=none, gRed] (6) at (6.75, 0) {$[C_4]$};
		\node [draw=none] (7) at (16, 4) {$[\Dic_3]$};
		\node [draw=none] (8) at (13, 4) {$[\Dic_3]$};
		\node [draw=none, gRed] (9) at (10, 4) {$[C_{12}]$};
		\node [draw=none, gRed] (10) at (6.75, 4) {$[Q_8]$};
		\node [draw=none] (11) at (11.25, 7.25) {$[\Dic_6]$};
		\node [draw=none, gRed] (34) at (11.25, 9.25) {$[\SL_2(\mathbb{F}_{13})]$};
		\draw[thin,black!30,->]  (0) to (1);
		\draw[thin,black!30,<-]  (2) to (0);
		\draw[thin,black!30,->]  (2) to (6);
		\draw[thin,black!30,->]  (2) to (5);
		\draw[thin,black!30,->]  (2) to (4);
		\draw[thin,black!30,->]  (1) to (3);
		\draw[thin,black!30,->]  (2) to (3);
		\draw[thin,black!30,->]  (3) to (7);
		\draw[thin,black!30,->]  (3) to (8);
		\draw[thin,black!30,->]  (3) to (9);
		\draw[thin,black!30,->]  (4) to (10);
		\draw[thin,black!30,->]  (5) to (10);
		\draw[thin,black!30,->]  (6) to (10);
		\draw[thin,black!30,->]  (4) to (7);
		\draw[thin,black!30,->]  (5) to (8);
		\draw[thin,black!30,->]  (6) to (9);
		\draw[thin,black!30,->]  (7) to (11);
		\draw[thin,black!30,->]  (8) to (11);
		\draw[thin,black!30,<-]  (11) to (9);
		\draw[thin,black!30,<-]  (11) to (10);
		\draw[thin,black!30,->]  (11) to (34);
		\draw[thin,black!30,->]  (10) to (34);
\end{tikzpicture}
\caption{The poset $D_G$ for $G = \SL_2(\mathbb{F}_{13})$. The objects of the sub-poset $I_G$ are highlighted in red.}\label{fig:sl2f13dg}
\end{figure}

	\bibliographystyle{alpha}
	\bibliography{mfgroups.bib}
	
\end{document}